\documentclass[a4paper]{article}

\usepackage{cite}

\usepackage{hyperref} 

\usepackage{amssymb,amsmath,amsthm}
\usepackage[english]{babel}

\usepackage{todonotes}
\usepackage{verbatim} 
\usepackage{enumerate}
\usepackage{mathtools}

\usepackage{bbm}

\newcommand{\cA}{\mathcal{A}}
\newcommand{\cB}{\mathcal{B}}
\newcommand{\cC}{\mathcal{C}}
\newcommand{\cD}{\mathcal{D}}
\newcommand{\cE}{\mathcal{E}}

\newcommand{\cL}{\mathcal{L}}

\newcommand{\cP}{\mathcal{P}}

\newcommand{\bQ}{\mathbb{Q}}
\newcommand{\bR}{\mathbb{R}}

\newcommand{\PR}{\mathbb{P}}

\newcommand{\bONE}{\mathbbm{1}}

\newcommand{\dd}{ \mathrm{d}}

\renewcommand{\epsilon}{\varepsilon}

\newcommand{\vn}[1]{\left| \! \left| #1\right| \! \right|}

\newcommand{\ip}[2]{\langle #1,#2\rangle}

\numberwithin{equation}{section}

\newtheorem{theorem}{Theorem}[section]
\newtheorem{lemma}[theorem]{Lemma}
\newtheorem{proposition}[theorem]{Proposition}

\theoremstyle{definition}
\newtheorem{definition}[theorem]{Definition}

\newtheorem{remark}[theorem]{Remark}

\newtheorem{assumption}[theorem]{Assumption}

\begin{document}

\title{A large deviation perspective on exponential decay of entropy and lower bounds on the Ricci-curvature}

\author{
\renewcommand{\thefootnote}{\arabic{footnote}}
Richard C. Kraaij
\footnotemark[1]
}

\footnotetext[1]{
Delft Institute of Applied Mathematics, Delft University of Technology, Mekelweg 4, 
2628 CD Delft, The Netherlands, E-mail: \texttt{kraaij.rc@gmail.com}.
}

\maketitle

\begin{abstract}
We offer a new point of view on the (Modified) Log-Sobolev inequality and lower bounds on the Ricci-curvature in the setting where the dynamics are obtained as the limit of Markov processes. In this setting, the large deviation rate function of the stationary measures of the Markov processes, plays the role of entropy. We define an \textit{entropy-information inequality (EII)} that generalizes the (MLSI) and is equivalent to exponential decay of the rate function along the flow, and define an \textit{entropy-convexity inequality (ECI)} that serves as an analogue of a lower bound on the Ricci-curvature in this setting.
\end{abstract}

 



\section{Introduction}

Log-Sobolev (LSI) and Modified Log-Sobolev (MLSI) inequalities attract considerable attention because of their connection to the exponential decay of entropy, and as a consequence the exponential decay of variance along the flow of the Kolmogorov forward equation.

More recently, also $\kappa$-lower bounds on the Ricci-curvature, in terms of $\kappa$-convexity of the entropy along displacement interpolations, have attracted interest. Such lower bounds, for $\kappa > 0$, imply among other things the (LSI) and (MLSI) inequalities, see cf. \cite{vRSt05,ErMa12,LoVi09,BaGeLe14}.

\smallskip

In this paper, we offer a new point of view on these notions in the setting where the flow is obtained as the limit of Markov processes, and where the entropy is replaced by the large deviation rate function, denoted by $S$, of the stationary measures of the Markov processes. This latter choice is motivated by Boltzmann's H-theorem which tells us that $S$ is a Lyapunov function for the flow \cite{RMN06}. We define an \textit{entropy-information inequality (EII)} that generalizes the (MLSI) and is a local equivalent to the global exponential decay of $S$ along the flow, and define an \textit{entropy-convexity inequality (ECI)} that serves as an analogue of a lower bound on the Ricci-curvature in this setting.

All our calculations and definitions are for finite-dimensional systems to make the intuitive ideas as clear as possible. The main definitions and results, however, generalize to the infinite dimensional setting, at the cost of greater technical difficulty. Examples of settings where the finite-dimensional ideas can be applied include the flow of the magnetization under high-temperature Glauber dynamics on the Curie-Weiss model and the flow of the densities of species for the Wright-Fisher diffusion, flows for which the exponential decay of entropy is not covered by the (LSI) or (MLSI) inequalities. Additionally, we consider the flow obtained from the non-reversible underdamped Langevin dynamics via vanishing diffusion constant, for which (EII) also works, but for which our analogue of the lower bound on the Ricci-curvature does not hold

\smallskip

The (EII) is based on identifying via Gronwall's lemma sufficient and necessary conditions for the exponential decay of entropy, that is, the exponential decay of the rate function $S$. A notion that generalizes a lower bound on the Ricci-curvature needs more work. 

It has been shown by \cite{vRSt05} that a lower bound on the Ricci-curvature, which is defined in geometric terms in the theory of differentiable manifolds, is equivalent to $\kappa$-convexity of the entropy along displacement interpolations. This equivalence inspired \cite{Ma11,ErMa12,Mi13} to define a lower bound on the Ricci-curvature for the space of measures on a finite set via a set of `displacement interpolations' obtained in a clever way from a Markov jump process generator, see also \cite{FaMa15,ErMaTe15} where various non-trivial discrete systems are being considered. Additionally, it was shown that the flow of the Kolmogorov forward equation corresponding to this generator is a gradient flow of the entropy with respect to a displacement distance obtained from the introduced class of interpolations.

\cite{Le13} has suggested to use entropic interpolations instead of displacement interpolations and we will generalize and use his notion in this paper. Even though calculations with entropic interpolations are generally harder, they do have some interesting features. First of all, the definition of entropic-interpolations is context-independent and only uses the underlying Markovian structure. Second, entropic interpolations are more regular. Third, displacement interpolations can be obtained as limits of entropic interpolations \cite{Le12,Le16} and thus entropic interpolations seem to be more general.

Our generalization of L\'{e}onard's notion of entropic interpolation is based on the observation that entropic interpolations minimise the path-space large deviation cost to connect two probability measures at two different times. The use of path-space large deviations also seems to be natural in view of the connections that have been observed between large deviations and gradient-flow systems, cf. \cite{MPR13,AdDiPeZi13,ErMaRe15} or references therein. These interpolations allow us to define a time-dependent `distance', which in turn can be used to define a notion of a lower-bound on the Ricci-curvature via convexity of $S$ along interpolations. An important feature is that if this non-standard lower bound holds with a positive constant, then we also have exponential decay of entropy with the same constant.

\smallskip

The paper is organized as follows. In Section \ref{section:large_deviations}, we introduce the path-space large deviation principles and how to find a suitable Lyapunov function or entropy. In Section \ref{section:decay_of_entropy}, we define the \textit{entropy-information} inequality which is a natural extension of the (LSI) and (MLSI) inequalities, given in terms of the large deviation principles. In Section \ref{section:entropic_interpolations}, we extend L\'{e}onards definition of entropic interpolations to the general setting of path-space large deviation principles. We study the behaviour of these interpolations and define a notion of a lower bound on the Ricci-curvature: the \textit{entropy-convexity inequality}.

In Section \ref{section:tensorization}, we show that the entropy-information and entropy-convexity inequalities perform well under tensorization. Finally, in Section \ref{section:examples}, we study five examples including the three mentioned above.

\section{Preliminaries: Large deviations and the McKean-Vlasov equation} \label{section:large_deviations}

Given a closed subset $E$ of $\bR^d$, we assume the existence of a sequence of measures $\PR_n \in \cP(D_E(\bR^+))$ so that the large deviation principle, for definitions see eg. \cite{DZ98,FK06}, holds for the trajectories:
\begin{equation} \label{eqn:LDP}
\PR_n\left[\{x(t)\}_{t \geq 0} \approx \{\gamma(t)\}_{t \geq 0}\right] \approx e^{-n u(\gamma)}.
\end{equation}
The topology on $D_E(\bR^+)$ is assumed to be the usual Skorokhod topology, see Chapter 3 in \cite{EK86}. We assume that the \textit{rate function} $u : D_E(\bR^+) \rightarrow [0,\infty]$ has compact level sets and has the form
\begin{equation*}
u(\gamma) = \begin{cases}
u_0(\gamma(0)) + \int_0^\infty \cL(\gamma(s),\dot{\gamma}(s)) \dd s & \text{if } \gamma \in \cA \cC, \\
\infty & \text{otherwise},
\end{cases}
\end{equation*}
where the \textit{Lagrangian} $\cL : E \times \bR^d \rightarrow [0,\infty)$ is lower semi-continuous and for each $x \in E$, $v \mapsto \cL(x,v)$ is convex. Finally, the space $\cA\cC \subseteq D_E(\bR^+)$ is the space of absolutely continuous curves.

\begin{definition}
We say that $\gamma \in D_E(\bR^+)$ is \textit{absolutely continuous}, $\gamma \in \cA\cC$, if for all functions $f$ that have a continuously differentiable extension to an open neighbourhood of $E$ we have for all $T > 0$
\begin{enumerate}[(a)]
\item $\int_0^T |\ip{\nabla f(\gamma(s))}{\dot{\gamma}(s)}| \dd s < \infty$,
\item $f(\gamma(T)) - f(\gamma(0)) = \int_0^T \ip{\nabla f(\gamma(s))}{\dot{\gamma}(s)} \dd s$.
\end{enumerate}
\end{definition}

Note that if $E = \bR$ this definition coincides with the usual definition of absolute continuity.

\smallskip

Such large deviation principles are obtained in various contexts, notably in Freidlin-Wentzell theory, which also works for Levy processes and for discrete time random walks, cf. Mogulskii's theorem, see \cite{FW98,FK06,DZ98} and references therein, but also for interacting jump processes on finite state spaces in \cite{Kr16b,DuRaWu16} and the Wright-Fisher model for population dynamics \cite{DaFe98}.

\smallskip 

In the setting that the measures $\PR_n$ correspond to Markov processes for which there exists stationary measures $\mu_n$ that satisfy the large deviation principle with rate function $S$, this $S$ is a Lyapunov function for the \textit{McKean-Vlasov} equation $\dot{x} = H_p(x,0)$. Here $H$ is the the \textit{Hamiltonian}, which is defined as $H(x,p) = \inf_v \left\{ \ip{p}{v} - \cL(x,v) \right\}$ and $H_p$ denotes the vector of derivatives of $H$ in the second coordinate. To be precise, it was found in \cite{RMN06} that $S(x(t)) \leq S(x(0))$ for any $t \geq 0$. 

In this paper, we analyse the decay of the entropy $S$ along the flow of the McKean-Vlasov equation in more detail. We will give conditions for exponential decay 
\begin{equation*}
S(x(t)) \leq e^{-\kappa t} S(x(0)).
\end{equation*}

Afterwards, we will extend the definition of entropic-interpolations introduced in \cite{Le13} and give conditions for the convexity of the entropy along these entropic interpolations.

\smallskip

The standing assumption on $H$ and $S$ for the results in this paper are the following. 

\begin{assumption} \label{assumption:basic_assumption}
We assume that the Hamiltonian $H : E \times \bR^d \rightarrow \bR$ satisfies
\begin{enumerate}[H(a)]
\item $H$ is twice continuously differentiable,
\item for every $x \in E$, the map $p \mapsto H(x,p)$ is convex and for every $x$ in the interior of $E$, the map $p \mapsto H(x,p)$ is strictly convex.
\end{enumerate}
There exists a continuous function $S : E \rightarrow [0,\infty]$ such that
\begin{enumerate}[S(a)]
\item $S$ is twice continuously differentiable on the interior $E^\circ$ of $E$,
\item for $x \in E^\circ$, we have $H(x,DS(x)) = 0$,
\item $S$ is a Lyapunov function for the McKean-Vlasov equation: if $x(t)$ solves $\dot{x}(t) = H_p(x(t),0)$ then $S(x(t)) \leq S(x(s))$ for all $0 \leq s \leq t$.
\end{enumerate}
\end{assumption}

\begin{remark}
The assumption that $H$ and $S$ are twice continuously differentiable can be relaxed to once continuously differentiable in various situations.
\end{remark}

The following three results verify that in the setting where $\PR_n$ are Markovian, the stationary measures satisfy the large deviation principle with rate function $S$ and where $S$ is differentiable on the interior of $E$, the conditions on $S$ of the assumption above are satisfied.  

First of all, we have an analogue of Boltzmann's H-theorem, which relates the stationary entropy $S$ to the solutions of the McKean-Vlasov equation.

\begin{proposition}[Proposition 3.1 in \cite{RMN06}, Proposition 2.8 in \cite{Kr16b}] \label{proposition:Htheorem2}
Suppose the measures $\PR_n$ correspond to Markov processes for which there exists stationary measures $\mu_n$ that satisfy the large deviation principle with rate function $S$.

Let $\{x(t)\}_{t \geq 0}$ be a solution to the McKean-Vlasov equation $\dot{x} = H_p(x,0)$, then  $S(x(t)) \leq S(x(s))$ for all $0 \leq s \leq t$.
\end{proposition}

In the following proposition, we show that $S$ is a solution to $Hf = 0$ in the viscosity sense. We will not use the theory of viscosity solutions at any other point in the paper, so we refer for definitions to \cite{CIL92} and \cite{CaSi04}.

\begin{lemma} \label{lemma:LDP_gives_entropy}
Suppose the measures $\PR_n$ correspond to Markov processes for which there exists stationary measures $\mu_n$ that satisfy the large deviation principle with rate function $S$. Then $S$ is a viscosity solution to $HS = 0$.
\end{lemma}

\begin{proof}
By a standard argument using dynamic programming, cf. Theorem 6.4.5 in \cite{CaSi04}, we find that for any function $u_0$, the function
\begin{equation*}
u(x,t) = \inf_{\substack{\gamma \in \cA\cC \\ \gamma(t) = x}} u_0(\gamma(0)) + \int_0^t \cL(\gamma(s),\dot{\gamma}(s)) \dd s
\end{equation*}
is a viscosity solution of
\begin{equation*}
\frac{\dd}{\dd t} u(x,t) + H(x, \nabla u(x,t)) = 0
\end{equation*}
on $E \times (0,T)$ where the gradient of $u$ is taken in the spatial dimensions. In the case that $u_0 = S$ is the large deviation rate function of the stationary measures, it follows that $u(\cdot,t) = S$ for all $t \geq 0$ by the contraction principle. As a direct consequence, we find that $S$ is a viscosity solution of $H(x,\nabla S(x)) = 0$ on $E$.
\end{proof}

As a consequence of the definition of viscosity solutions, we obtain that the equation is satisfied in any point where a viscosity solution is differentiable.

\begin{lemma}
Suppose the measures $\PR_n$ correspond to Markov processes for which there exists stationary measures $\mu_n$ that satisfy the large deviation principle with rate function $S$. If $x$ is in the interior of $E$ and is such that $S$ is differentiable at $x$, then $H(x,DS(x)) = 0$.
\end{lemma}

\section{Exponential decay of entropy} \label{section:decay_of_entropy}

We start by studying the decay of $S$ along the solutions of the McKean-Vlasov equation. Motivated by the analogous quantities in the theory of (modified) logarithmic Sobolev inequalities, we define the concept of information.

\begin{definition}
Let $H$ and $S$ satisfy Assumption \ref{assumption:basic_assumption}. We define the \textit{information} $I : E^\circ \rightarrow \bR^+$ by
\begin{equation*}
I(x) = - \ip{DS(x)}{H_p(x,0)}.
\end{equation*}
We say that $H$ and $S$ satisfy an \textit{entropy-information inequality} with constant $\kappa > 0$, denoted by (EII)($\kappa$), if for all $x \in E^\circ$:
\begin{equation*}
\kappa S(x) \leq  I(x).
\end{equation*}
\end{definition}

Note that as $I(x) = - \frac{\dd}{\dd t} |_{t = 0} S(x(t))$ for the solution $x(t)$ to the McKean-Vlasov equation with $x(0) = x$, it follows that $I(x) \geq 0$ by Proposition \ref{proposition:Htheorem2}.

\smallskip

This entropy-information inequality is a naturally connected to similar inequalities present in the literature. In the setting of the measure-valued flow generated by the Kolmogorov forward equation of a diffusion operator, the derivative of the entropy along the flow is called the \textit{Fisher information}. Thus, the entropy-information inequality is related to the well known Log-Sobolev Inequality, we refer to Section 5.2 in \cite{BaGeLe14}. For the measure valued flow generated by the Kolmogorov forward equation of a Markov jump process, cf. \cite{CDP09}, the $2\kappa$-entropy-information inequality coincides with the modified logarithmic Sobolev inequality with constant $\kappa$, as we show below.

\begin{remark}
Let $F = \{1,\dots,d\}$, set $E = \cP(F) = \{x \in \bR^d \, | \, \sum x_a = 1, x_a \geq 0\}$. Let $Af(a) = \sum_b r(a,b) \left[f(b) - f(a)\right]$ be the generator of a reversible Markov jump process. In other words, there is a measure $\pi \in E$ such that $\pi_a r(a,b) = \pi_b r(b,a)$. Denote by $\cE$ the Dirichlet form
\begin{equation*}
\cE(f,g) = \frac{1}{2}\sum_{a,b} \pi_a r(a,b)\left[f(b) - f(a)\right]\left[g(b) - g(a)\right],
\end{equation*}
and by $Ent(f) = \sum_a \pi_a f(a) \log f(a)$.

Let $f : F \rightarrow [0,\infty)$ such that $f(a) > 0$ for all $a$ and $\int f \dd \pi = 1$. Then the Modified Logarithmic Sobolev Inequality with constant $\kappa$, is given by
\begin{equation*}
\kappa Ent(f) \leq \frac{1}{2} \cE(f,\log f).
\end{equation*}
On the other hand, the Hamiltonian of the system obtained via the large deviations of the empirical density of independent copies of the Markov process is given by
\begin{equation*}
H(x,p) = \sum_{a,b} x_a r(a,b) \left[e^{p_b - p_a} - 1 \right],
\end{equation*}
and the entropy $S$ is given by $S(x) = \sum_a x_i \log \frac{x_i}{\pi_i}$.

\smallskip

Any $f : F \rightarrow [0,\infty)$ such that $\int f \dd \pi = 1$ induces a measure $x \in E$ by $x_i = f(i) \pi_i$. This way, it is immediately clear that
\begin{equation*}
S(x) = Ent(f).
\end{equation*}
On the other hand the equality $I(x) = \cE(f,\log f)$ follows by a straightforward calculation using the reversibility of $\pi$. We conclude that the MLSI($\kappa$) is equivalent EII($2\kappa$).
\end{remark}

To use the entropy-information inequality for any solution of the McKean-Vlasov equation, we assume additionally that solutions enter the interior of $E$ immediately.

\begin{assumption} \label{assumption:McKean_Vlasov_in_interior}
Let $\{x(t)\}_{t \geq 0}$ solve $\dot{x}(t) = H_p(x(t),0)$, $x(0) = x_0 \in E$. Then for all $t > 0$, we have $x(t) \in E^\circ$.
\end{assumption}

By analogy, the following result is to be expected.

\begin{lemma} \label{lemma:EII_equiv_with_entropy_decay}
Let $H$ and $S$ satisfy Assumptions \ref{assumption:basic_assumption} and \ref{assumption:McKean_Vlasov_in_interior}. 
$H$ satisfies EII($\kappa$) if and only if for all $x_0$ and solutions $\{x(t)\}_{t \geq 0}$ to $\dot{x}(t) = H_p(x(t),0)$ with $x(0) = x_0$ we have
\begin{equation*}
S(x(t)) \leq e^{-\kappa t} S(x(0)).
\end{equation*}
\end{lemma}

\begin{proof}
Suppose that $H$ satisfies EII($\kappa$). Consider any trajectory $\{x(t)\}_{t \geq 0}$ that is fully in $E^\circ$. We find by EII($\kappa$) that for every $t \geq 0$
\begin{equation*}
\frac{\dd}{\dd t}S(x(t)) = \ip{DS(x(t))}{H_p(x(t),0)} = -I(x) \leq -\kappa S(x(t)).
\end{equation*}
It follows by Grönwall's inequality that
\begin{equation*}
S(x(t)) \leq e^{-\kappa t} S(x(0)).
\end{equation*}
If $x(0) \notin E^\circ$, then by considering the trajectory started from $x(\varepsilon)$, taking $\varepsilon$ to $0$, the result follows by continuity of $S$.

\smallskip

For the reverse inequality, pick a point $x_0$ and a solution to the McKean-Vlasov equation started from $x_0$. Then $\kappa S(x_0) \leq I(x_0)$ follows by differentiation.
\end{proof}

It is well known that control on the second derivative of the entropy along solutions of the McKean-Vlasov equation yields stronger control on the decay of the entropy, see for example Lemma 2.1 in \cite{CDP09}. The second derivative of the entropy of $S$ gives:
\begin{multline} \label{eqn:second_derivative_entropy}
\frac{\dd^2}{\dd t^2} S(x(t))  = \frac{\dd}{\dd t} \ip{DS(x(t))}{H_p(x(t),0)} \\
= \ip{D^2S(x(t))H_p(x(t),0)}{H_p(x(t),0)} + \ip{DS(x(t))}{H_{px}(x(t),0) H_p(x(t),0)}.
\end{multline}
Note that we write $H_{px}(x,p)$ for the matrix
\begin{equation*}
H_{px}(x,p) \begin{bmatrix} H_{p_1,x_1}(x,p) & \dots & H_{p_1,x_d}(x,p) \\ \vdots & \ddots & \vdots \\
H_{p_d,x_1(x,p)} & \dots & H_{p_d,x_d}(x,p)\end{bmatrix}.
\end{equation*}

We obtain the following result, giving an inequality that implies EII($\kappa$) if there is only one attracting stationary point.

\begin{proposition} \label{proposition:EII_entropy_decay}
Let $H$ and $S$ satisfy Assumptions \ref{assumption:basic_assumption} and \ref{assumption:McKean_Vlasov_in_interior}. Let $\{x(t)\}_{t \geq 0}$ be a solution to the McKean-Vlasov equation $\dot{x} = H_p(x,0)$. Then the following two statements are equivalent.
\begin{enumerate}[(a)]
\item For all $x \in E^\circ$, we have
\begin{equation} \label{eqn:second_derivative_inequality_McKean_Vlasov}
\kappa I(x) \leq  \ip{DS(x)}{H_{px}(x,0) H_p(x,0)} + \ip{D^2 S(x)H_p(x,0)}{H_p(x,0)}.
\end{equation}
\item For all solutions $\{x(t)\}_{t \geq 0}$ of the McKean-Vlasov equation, we have
\begin{equation*}
\frac{\dd}{\dd t} S(x(t)) \geq e^{-\kappa t} \frac{\dd}{\dd t} S(x(t)) |_{t = 0}, \quad \text{ and } \quad I(x(t)) \leq e^{-\kappa t}  I(x(0))
\end{equation*}
\end{enumerate}

Suppose that $S$ is bounded from below and that \eqref{eqn:second_derivative_inequality_McKean_Vlasov} is satisfied. Let $S_\infty := \lim_{t \rightarrow \infty} S(x(t))$, (which exists as $S$ is decreasing along solutions of the McKean-Vlasov equation), then
\begin{equation} \label{eqn:extended_EII}
S(x(t)) - S_\infty \leq e^{-\kappa t} \left(S(x(0)) - S_\infty\right).
\end{equation}
\end{proposition}

\begin{remark}
If $S(x)$ is convex, $D^2S(x)$ is a positive operator. Thus, a weaker criterion for the exponential decay of entropy is given by
\begin{equation} \label{eqn:reduction:EII_for_convex_S}
\kappa I(x) \leq   \ip{DS(x)}{H_{px}(x,0) H_p(x,0)}.
\end{equation}
\end{remark}

\begin{proof}[Proof of Proposition \ref{proposition:EII_entropy_decay}]
By \eqref{eqn:second_derivative_entropy}, we note that \eqref{eqn:second_derivative_inequality_McKean_Vlasov} can be rewritten as
\begin{equation*}
\frac{\dd}{\dd t} I(x(t)) \leq - \kappa I(x(t)) 
\end{equation*}
As in the proof of Lemma \ref{lemma:EII_equiv_with_entropy_decay}, we obtain the equivalence of (a) and (b).

We proceed with the proof of \eqref{eqn:extended_EII}. First, we integrate from $t$ to $T$ the inequality
\begin{equation*}
\kappa \frac{\dd}{\dd t} S(x(t)) = - \kappa I(x(t)) \geq \frac{\dd}{\dd t} I(x(t)),
\end{equation*}
which yields
\begin{equation*}
\kappa\left[S(x(T)) - S(x(t))\right] \geq  \left[I(x(T)) - I(x(t))\right].
\end{equation*}
As $T \rightarrow \infty$, we find by the first claim of the Lemma that $I(x(T)) \rightarrow 0$. Additionally, as $S$ is decreasing along the solutions of the McKean-Vlasov equation $S(x(T)) \rightarrow S_\infty$. We conclude that
\begin{equation*}
\kappa\left[S(x(t)) - S_\infty\right] \leq I(x(t)).
\end{equation*}
The claim follows as in the proof of Lemma \ref{lemma:EII_equiv_with_entropy_decay}.
\end{proof}

To prove that constants are optimal for examples we will consider below, we have the following two results, the first of which is immediate from the results above.

\begin{lemma} \label{lemma:optimality_converse_inequality}
Let $H$ and $S$ satisfy Assumptions \ref{assumption:basic_assumption} and \ref{assumption:McKean_Vlasov_in_interior}. Let $\{x(t)\}_{t \geq 0}$ be a solution to the McKean-Vlasov equation $\dot{x} = H_p(x,0)$ such that $x(t) \rightarrow x_s$, where $x_s$ is a local minimum of $x \mapsto S(x)$. 

Suppose that in a neighbourhood $U$ of $x_s$, there is an $\kappa > 0$ such that we have the inequality
\begin{equation} \label{eqn:second_derivative_inequality_McKean_Vlasov_other_side}
\kappa S(x) > I(x).
\end{equation}
Then, if $t_0$ is large enough such that $x(t) \in U$ for $t \leq t_0$, we have
\begin{equation*} 
S(x(t)) - S(x_s) > e^{-\kappa t} \left(S(x(t_0)) - S(x_s)\right)
\end{equation*}
for all $t \leq t_0$.
\end{lemma}

\begin{proposition}\label{proposition:optimality_converse_inequality_via_matrix}
Let $H$ and $S$ satisfy Assumptions \ref{assumption:basic_assumption} and \ref{assumption:McKean_Vlasov_in_interior}. Furthermore, let $x_s$ be a local minimum of $S$ and stationary: $H_p(x_s,0) = 0$. Assume that $D^2S(x_s)$ is strictly positive definite. 

If $c > 0$ is such that the matrix
\begin{equation*}
c\bONE + 2 H_{px}(x_s,0)
\end{equation*}
is strictly positive definite, than $c$ is an upper bound for the entropy-information inequality. In other words, if $EII(\kappa)$ holds, then $\kappa < c$.
\end{proposition}

\begin{proof}
As $x_s$ is a stationary point and a local minimum of $S$, we find $H_p(x_s,0) = 0$ and $DS(x_s) = 0$. Thus, we obtain
\begin{equation*}
D^2 I(x_s) = -2 D^2 S(x_s) H_{px}(x_s,0).
\end{equation*}
A Taylor expansion of $c S - I$ around $x = x_s$ yields
\begin{equation*}
(cS-I)(x) = (x-x_s) D^2S(x_s) \left(c\bONE - 2 H_{px}(x_s,0)\right)(x-x_s) + o(|x-x_s|^2).
\end{equation*}
Because $D^2S(x_s)$ and $c\bONE + 2H_{px}(x_s,0)$ are strictly positive-definite, we can find a neighbourhood of $x_s$ on which we can apply Lemma \ref{lemma:optimality_converse_inequality}, which proves the claim.
\end{proof}

\section{Entropic interpolations} \label{section:entropic_interpolations}

Based on the equivalence in the smooth differential geometric setting of a lower bound on the Ricci curvature and the $\kappa$-convexity of the entropy along displacement interpolations, the notion of a lower bound on the Ricci-curvature can be defined in any setting that allows for displacement interpolations and an entropy.

Additionally, in the context of the log-Sobolev inequality, cf. \cite{BaGeLe14}, and in the context of the modified log-Sobolev inequality, cf. \cite{ErMa12}, it is well known that lower bounds on suitably chosen notions of Ricci curvature imply exponential decay of information, and thus entropy.

As it is not clear how to define displacement interpolations in our setting, we need to introduce a notion of interpolations to obtain a similar result for the entropy-information inequality.

\smallskip

L\'{e}onard \cite{Le13,Le12,Le16} introduced and studied entropic interpolations and has shown that displacement interpolations can be obtained as limits of entropic interpolations.

This indicates that entropic interpolations can serve as a generalization of displacement interpolations for the study of lower bounds on Ricci-curvature. In the context where the Hamiltonian corresponds to the large deviation behaviour of the trajectories of the empirical density of independent copies of a process, i.e. \cite{Kr14} L\'{e}onard \cite{Le13} defines an entropic interpolation between $\pi$ and $\nu$ in time $T$ in terms of an $(f,g)$ transform. Using the connection of the $(f,g)$ transform to solutions of the Schr\"{o}dinger problem in \cite[Theorem 3.3]{Le14}, this transform corresponds to the trajectory of measures $\{\mu(t)\}_{0 \leq t \leq T}$, where $\mu(t) := \bQ^*_t$ is the law of $X(t)$ under $\bQ^*$, and where $\bQ^*$ minimizes
\begin{equation*}
\inf \left\{H(\bQ \, | \, \PR) \, \middle| \, \bQ_0 = \pi, \bQ_T = \nu \right\},
\end{equation*}
where $H$ is the relative entropy. This minimization problem can be re-expressed in terms of the path-space large deviation problem of the trajectory of the empirical distribution of independent copies. This re-formulation of the minimization problem generalizes to interacting systems and motivates the following definition.

\begin{definition}
We say that an absolutely continuous trajectory $\gamma^* : [0,T] \rightarrow E$ is an entropic interpolation between $x$ and $y$ in time $T$ if $\gamma^*(0) = x$, $\gamma^*(T) = y$ and
\begin{equation*}
\int_0^T \cL(\gamma^*(s),\dot{\gamma}^*(s)) \dd s = \inf_{\substack{\gamma \in \cA\cC \\ \gamma(0) = x, \gamma(T) = y}} \int_0^T \cL(\gamma(s),\dot{\gamma}(s)) \dd s.
\end{equation*}
\end{definition}

This definition also has connections to weak-KAM theory, cf. \cite{Fa08,CoIt99}, the quantity 
\begin{equation*}
h_t(x,y) := \inf_{\substack{\gamma \in \cA\cC \\ \gamma(0) = x, \gamma(T) = y}} \int_0^T \cL(\gamma(s),\dot{\gamma}(s)) \dd s,
\end{equation*}
is sometimes called the finite-time potential and is used to define the well-studied Peirls barrier and the Ma\~{n}\'{e} potential. The entropic interpolation in this context is sometimes called a Tonelli-minimizer. 

\smallskip

We will make the following assumption in this section, which is necessary in the case that $S$ is not differentiable on the boundary of $E$. In the setting of one-dimensional reversible processes, we will give conditions under which this assumption is always satisfied.

\begin{assumption} \label{assumption:entropic_interpolations_not_on_boundary}
Any entropic interpolation $\gamma : [0,T] \rightarrow E$ is such that for all $t \in (0,T)$ we have $\gamma(t) \in E^\circ$.
\end{assumption}

Note that this assumption implies Assumption \ref{assumption:McKean_Vlasov_in_interior}.

Consider the Hamilton equations:
\begin{equation} \label{eqn:Hamilton_equations_interpolations}
\begin{bmatrix}
\dot{x} \\ \dot{p} 
\end{bmatrix}
= 
\begin{bmatrix} H_p(x,p) \\ - H_x(x,p)\end{bmatrix}. 
\end{equation}

\begin{lemma} \label{lemma:entropic_interpolations_solve_Hamilton_equations}
Let $H$ and $S$ satisfy Assumption \ref{assumption:basic_assumption} and let $x$ be an entropic interpolation satisfying Assumption \ref{assumption:entropic_interpolations_not_on_boundary}. For $t \in (0,T)$, set $p(t) = \cL_v(x(t),\dot{x}(t))$. Then $(x(t),p(t))$ is twice continuously differentiable and solves the Hamilton equations for $t \in (0,T)$.
\end{lemma}

\begin{proof}
This result follows as in Theorems 6.2.8 and 6.3.3 in \cite{CaSi04} as we can work in the interior of $E$ due to Assumption \ref{assumption:entropic_interpolations_not_on_boundary}.
\end{proof}

The two components of the Hamilton equations take over the role of equations (14) and (15) in \cite{Le13}. The connection between the first component and the Hamilton equations is immediate, whereas for the second component, (15) in \cite{Le13} describes the evolution of $f$, whereas the second component of the Hamilton equations describes the evolution of $p(t) = \nabla f(x(t))$ along the trajectory of the first variable.

\smallskip

Note that the solution $x(t)$ to the McKean-Vlasov equation is always an entropic interpolation between $x(0)$ and $x(t)$ in time $t$ for any time $t \geq 0$. This corresponds to a solution of the Hamilton equations in which $p(t) = 0$ for all $t \geq 0$.

To study the entropic interpolations, we follow Section 2.7 in \cite{BDGJL02} and introduce the adjoint Hamiltonian.

\subsection{The adjoint Hamiltonian}

\begin{definition}
Let $H$ and $S$ satisfy Assumption \ref{assumption:basic_assumption}. We define the adjoint $H^*$ of $H$ with respect to $S$ for $x \in E^\circ$ and $p \in \bR^d$ by
\begin{equation*}
H^*(x,p) = H(x,DS(x) - p).
\end{equation*}
If $H$ is a Hamiltonian with entropy $S$, we say that $H$ is reversible with respect to $S$ if $H^* = H$.
\end{definition}

The adjoint Hamiltonian is is related to the time-reversal of trajectories, see Lemma \ref{lemma:adjoint_flow_solves_adjoint_Hamilton_equations} below. The terminology of a reversible Hamiltonian corresponds to the picture introduced in Lemma \ref{lemma:LDP_gives_entropy}. If $H$ is the Hamiltonian corresponding to a sequence of reversible processes, and $S$ is the corresponding entropy of the stationary and reversible measures, then $H$ will be reversible with respect to $S$.

\begin{remark}
Even tough it holds for most one-dimensional examples in this paper that $H = H^*$, a non reversible one-dimensional example is obtained by considering the large deviation behaviour of the average of $n$ independent Levy processes on $\bR$ with generator
\begin{equation*}
Af(x) = \frac{1}{2} f''(x) - (x + 1) f'(x) + f(x+1) - f(x),
\end{equation*}
which corresponds to a Hamiltonian of the form
\begin{equation*}
H(x,p) = \frac{1}{2}p^2 - (x+1)p + e^{p} -1.
\end{equation*}
\end{remark}

As in Section 2.7 of \cite{BDGJL02}, we can relate the adjoint Hamiltonian to the reversal of time.

\begin{lemma} \label{lemma:adjoint_flow_solves_adjoint_Hamilton_equations}
Let $H$ and $S$ satisfy Assumption \ref{assumption:basic_assumption}. Fix some time $T > 0$. The curve $(x(t),p(t))_{0 < t < T}$ solves the Hamilton equations for $H$ if and only if $(x^*(t), p^*(t))_{0 < t < T} := (x(T-t),DS(x(T-t)) - p(T-t))_{0 < t < T}$ solves the Hamilton equations for $H^*$.
\end{lemma}

\begin{proof}
Let $(x(t),p(t))_{0 < t < T}$ solve the Hamilton equations. First note that $H^*_p(x,p) = - H_p(x,DS(x) - p)$ by definition. We look at the derivative of $x^*(t)$:
\begin{align*}
\frac{\dd}{\dd t} x^*(t) & = \frac{\dd}{\dd t} x(T-t) \\
& = - \dot{x}(T-t) \\
& = - H_p(x(T-t),p(T-t)) \\
& = H^*_p(x(T-t),DS(x(T-t)) - p(T-t)) \\
& = H^*_p(x^*(t), p^*(t)).
\end{align*}
Secondly, we consider the derivative of $p^*(t)$:
\begin{align*}
\frac{\dd}{\dd t} p^*(t) & = \frac{\dd}{\dd t} p(T-t) \\
& = - \dot{p}(T-t) \\
& = H_x(x(T-t),p(T-t)) \\
& = H^*_x(x(T-t),DS(x(T-t)) - p(T-t)) \\
& = H^*_x(x^*(t),p^*(t)).
\end{align*}
So indeed $(x^*(t),p^*(t))_{0 < t < T}$ solve the Hamilton equations for $H^*$. The second implication of the lemma follows from the first one and the fact that $H^{**} = H$.
\end{proof}

Denote by $\overline{H}$ the infimal convolution of $H$ and the time inverse of $H^*$.
\begin{equation*}
\overline{H}(x,p) = 2 H\left(x, \frac{1}{2}p + \frac{1}{2} DS(x)\right) = \inf\left\{H(x,p-q) + H^*(x,-q) \, \middle| \, q \in \bR^d \right\}.
\end{equation*}
It follows that
\begin{equation*}
\sup_p \ip{p}{v} - \overline{H}(x,p) = \cL(x,v) + \cL^*(x,-v).
\end{equation*}

Appropriately changing the momentum of a entropic interpolation yields an solution to the Hamilton equations for $\overline{H}$.

\begin{lemma} \label{lemma:symmetrized_flow_solves_adjoint_Hamilton_equations}
Let $H$ and $S$ satisfy Assumption \ref{assumption:basic_assumption}. Fix some time $T > 0$. The curve $(x(t),p(t))_{0 < t < T}$ solves the Hamilton equations for $H$ if and only if $(x(t),2p(t) - DS(x(t)))_{0 < t < T}$ solves the Hamilton equations for $\overline{H}$.
\end{lemma}

\begin{proof}
Note that for all $x \in E$ and $p \in \bR^d$, we have $\overline{H}_p(x,q) = H_p\left(x,\frac{1}{2}q + \frac{1}{2}DS(x) \right)$. This establishes that $H_p(x(t),p(t)) = \overline{H}_p(x(t),2p(t) - DS(x(t)))$.

\smallskip

For the equivalences of the evolution of the momenta, an elementary calculation yields for all $x \in E^\circ$ and $q \in \bR^d$ that
\begin{equation*}
\overline{H}_x(x,q) = 2 H_x\left(x,\frac{1}{2}q + \frac{1}{2}DS(x) \right) + H_p\left(x,\frac{1}{2}q + \frac{1}{2}DS(x) \right) D^2 S(x).
\end{equation*}
This gives for all $t \in (0,T)$ that
\begin{equation*}
- \overline{H}_x(x(t),2p(t) - DS(x(t))) = 2 H_x\left(x(t),p(t) \right) + H_p\left(x(t),p(t) \right) D^2 S(x(t)).
\end{equation*}
On the other hand, we have
\begin{equation*}
\frac{\dd}{\dd t} (2p(t) - DS(x(t)) = \frac{\dd}{\dd t} 2p(t) - H_p(x(t),p(t)) D^2S(x(t)).
\end{equation*}
Combining these last two equations yields that $\dot{p}(t) = - H_x(x(t),p(t))$ if and only if $\frac{\dd}{\dd t} (2p(t) - DS(x(t))) = - \overline{H}_x(x(t),2p(t) - DS(x(t)))$.
\end{proof}

Denote by $\overline{\cL}(x,v) = \cL(x,v) + \cL^*(x,-v)$ the Lagrangian obtained from $\overline{H}$. The following proposition is immediate from the two previous lemmas.

\begin{proposition} \label{proposition:equivalence_entropic_interpolations}
Let $H$ and $S$ satisfy Assumption \ref{assumption:basic_assumption}. Fix some time $T > 0$ and a curve $(x(t))_{0 < t < T}$ between $x(0) = x_0$ and $x(T) = x_T$. The following are equivalent
\begin{enumerate}[(a)]
\item $x$ is an entropic interpolation from $x_0$ to $x_T$ in time $T$ for $\cL$,
\item $t \mapsto x(T-t)$ is an entropic interpolation from $x_T$ to $x_0$ in time $T$ for $\cL^*$.
\item $x$ is an entropic interpolation from $x_0$ to $x_T$ in time $T$ for $\overline{\cL}$.
\end{enumerate}
\end{proposition}

For reversible one-dimensional Hamiltonians, we give conditions under which Assumption \ref{assumption:entropic_interpolations_not_on_boundary} is always satisfied.

\begin{proposition} \label{proposition:trajectories_in_interior}
Suppose $E = [a,b]$ and $H$ and $S$ satisfy Assumption \ref{assumption:basic_assumption} and suppose that: 
\begin{enumerate}[(a)]
\item $H = H^*$.
\item $H_p(a,0) > 0$ and $H_p(b,0) < 0$.
\item The maps $x \mapsto \cL(x,0)$ and $x \mapsto S(x)$ are decreasing on an open neighbourhood $U_{a}$ of $a$ and increasing on an open neighbourhood $U_b$ of $b$.
\item We have
\begin{equation*}
\lim_{x \downarrow a} \frac{\cL(x,0) - \cL(a,0)}{S(x) - S(a)} = \infty, \quad
\lim_{x \uparrow b} \frac{\cL(x,0) - \cL(b,0)}{S(x) - S(b)} = \infty.
\end{equation*}
\end{enumerate}
Then all entropic interpolations $\{x(t)\}_{0 \leq t \leq T}$ satisfy $x(t) \in E^\circ$ for $t \in (0,T)$.
\end{proposition}

As the proof of this proposition is independent of the rest of the results, we postpone the proof until Section \ref{section:proof_entropic_interpolations_in_interior}.

\subsection{The evolution of entropy along an entropic interpolations}

Analogous to the definition of $\cL$, we define $\cL^*$ to be the Lagrangian corresponding to $H^*$, i.e. for $x \notin \partial E$, we set $\cL^*(x,v) = \sup_p \ip{p}{v} - H^*(x,p)$. The following result has been implicitly found in (2.2) of \cite{BDGJL02}.

\begin{lemma}\label{lemma:derivative_of_entropy}
Let $H$ and $S$ satisfy Assumption \ref{assumption:basic_assumption} and let $\gamma : [0,T] \rightarrow E$ be an absolutely continuous trajectory and let $t$ be a time at which $\gamma$ is differentiable and $\gamma(t) \in E^\circ$. Define the time-backward trajectory $\gamma^*(s) := \gamma(T-s)$. Then we have
\begin{equation*}
\frac{\dd}{\dd t} S(\gamma(t)) = \cL(\gamma(t),\dot{\gamma}(t)) - \cL^*(\gamma^*(T-t),\dot{\gamma}^*(T-t)).
\end{equation*}
In particular, if $H = H^*$, it follows that
\begin{equation*}
\frac{\dd}{\dd t} S(\gamma(t)) = \cL(\gamma(t),\dot{\gamma}(t)) - \cL(\gamma(t),-\dot{\gamma}(t)).
\end{equation*}
\end{lemma}

\begin{proof}
Set $p = \cL_v(\gamma(t),\dot{\gamma}(t))$ and $p^* = DS(\gamma(t)) - p$. We obtain
\begin{align*}
& \frac{\dd}{\dd t} S(\gamma(t))  = DS(\gamma(t)) \dot{\gamma}(t) \\
& = p H_p(\gamma(t),p) - H(\gamma(t),p) \\
& \qquad - \left[\left(p - DS(\gamma(t)) \right) H_p(\gamma(t),p) - H(\gamma(t),p)\right] \\
& = \cL(\gamma(t),\dot{\gamma}(t)) - \left(DS(\gamma(t) - p\right) H^*_p(\gamma(t),DS(\gamma(t)) - p) \\
& \qquad + H^*(\gamma(t),DS(x(t)) - p(t)) \\
& = \cL(\gamma(t),\dot{\gamma}(t)) - p^* H^*_p(\gamma^*(T-t),p^*) + H^*(\gamma^*(T-t),p^*) \\ 
& = \cL(\gamma(t),\dot{\gamma}(t)) - \cL^*(\gamma^*(T-t),\dot{\gamma^*}(T-t)),
\end{align*}
where we have used in the last line that 
\begin{equation*}
\dot{\gamma}^*(T - t) = - \dot{\gamma}(t) = - H_p(\gamma(t),p) = H_p^*(\gamma^*(T-t),p^*).
\end{equation*}
\end{proof}

Because an entropic interpolation $\{x(t)\}_{0 \leq t \leq T}$ gives rise to a twice continuously differentiable trajectory $(x,p)$ that solves the Hamilton equations, we see that for this trajectory Lemma \ref{lemma:derivative_of_entropy} holds for all times at which the trajectory is in the interior of $E$. We use this to study the behaviour of the entropy along the interpolation.

As the entropy along an arbitrary entropic interpolation is not expected to decrease, we directly study the second derivative of the entropy $S$ along an entropic interpolation $\{x(t)\}_{t \in [0,T]}$ satisfying Assumption \ref{assumption:entropic_interpolations_not_on_boundary}. In Lemma \ref{lemma:derivative_of_entropy}, we saw that the first derivative of $S$ contains a part involving $\cL$ and a part involving $\cL^*$. We first consider the part involving $\cL$. Note that for an entropic interpolation $\frac{\dd}{\dd t} H(x(t),p(t))  = 0$ by the Hamilton equations. For $t \in (0,T)$, we have
\begin{equation} \label{eqn:second_derivative_S_H}
\begin{aligned}
\frac{\dd}{\dd t} \cL(x(t),\dot{x}(t)) & = \frac{\dd}{\dd t} \left(\ip{p(t)}{H_p(x(t),p(t))} - H(x(t),p(t))\right) \\
& = - \ip{H_x(x(t),p(t))}{H_p(x(t),p(t))} \\
& \qquad + \ip{p(t)}{H_{px}(x(t),p(t)) H_p(x(t),p(t))} \\
& \qquad - \ip{p(t)}{H_{pp}(x(t),p(t))H_x(x(t),p(t))} 
\end{aligned}
\end{equation}
Set $x^*(t) := x(T-t)$ and $p^*(t) = DS(x(T-t)) - p(T-t)$. For the derivative of the second term, we obtain similarly that
\begin{equation}\label{eqn:second_derivative_S_Hstar}
\begin{aligned}
\hspace{-2em} & \frac{\dd}{\dd t} \left(- \cL^*(x^*(T-t),\dot{x}^*(T-t))\right) \\
\hspace{-2em} & = \frac{\dd}{\dd(T-t)} \cL^*(x^*(T-t),\dot{x}^*(T-t)) \\
\hspace{-2em} & = p^*(T-t) H_{px}^*(x^*(T-t),p^*(T-t)) \\
\hspace{-2em} & \quad  - p^*(T-t) H_{pp}(x^*(T-t),p^*(T-t)) H_x^*(x^*(T-t),p^*(T-t)) \\
\hspace{-2em} & \quad - H_x^*(x^*(T-t),p^*(T-t))H_p^*(x^*(T-t),p^*(T-t)).
\end{aligned}
\end{equation}

\begin{definition}
Let $H$ and $S$ satisfy Assumption \ref{assumption:basic_assumption}. We say that $H$ and $H^*$ satisfy the  $\kappa$-\textit{entropy-convexity inequality}, denoted by ECI($\kappa$) if for all $x \in E^\circ$ and $p \in \bR^d$, we have
\begin{align*}
& \kappa \left[\ip{p}{H_p(x,p)} - H(x,p)\right] + \kappa \left[\ip{p^*}{H_p^*(x,p^*)} - H^*(x,p^*)\right] \\
& \leq  \ip{p}{H_{px}(x,p)H_p(x,p)} - \ip{p}{H_{pp}(x,p)H_x(x,p)} \\ 
& \qquad -\ip{H_x(x,p)}{H_p(x,p)} + \ip{p^*}{H_{px}^*(x,p^*)H_p^*(x,p^*)} \\
& \qquad - \ip{p^*}{H_{pp}^*(x,p^*)H_x^*(x,p^*)} - \ip{H_x^*(x,p^*)}{H_p^*(x,p^*)},
\end{align*}
where $p^* = DS(x) - p$. If $H = H^*$, we say that $H$ satisfies the $\kappa$-\textit{entropy-convexity inequality} if
\begin{multline} \label{eqn:entropy_convexity_reversible}
\kappa\left[\ip{p}{H_p(x,p)} - H(x,p)\right] \leq \ip{p}{H_{px}(x,p)H_p(x,p)} \\
- \ip{p}{H_{pp}(x,p)H_x(x,p)} - \ip{H_x(x,p)}{H_p(x,p)},
\end{multline}
for all $x \in E^\circ$ and $p \in \bR^d$.
\end{definition}

It is immediate that if $H = H^*$ \eqref{eqn:entropy_convexity_reversible} implies that $H$ and $H^*$ satisfy the  $\kappa$-\textit{entropy-convexity inequality}.

For $T > 0$ let
\begin{equation*}
G_T(s,t) = \begin{cases}
\frac{s(T-t)}{T} & \text{if } s \leq t, \\
\frac{t(T-s)}{T} & \text{if } s \geq t.
\end{cases}
\end{equation*}

A direct computation for a function $\phi \in C([0,T])$ that is twice continuously differentiable on $(0,T)$ that
\begin{align*}
\phi(t) = \frac{T-t}{T} \phi(0) + \frac{t}{T} \phi(T) - \int_0^T \ddot{\phi}(s)G_T(s,t) \dd s
\end{align*}

Combining \eqref{eqn:second_derivative_S_H} and \eqref{eqn:second_derivative_S_Hstar} with the definition of the entropy convexity inequality, we have the following result. 

\begin{theorem} \label{theorem:convexity_of_entropy_along_interpolations}
Let $H$ and $S$ satisfy Assumption \ref{assumption:basic_assumption}. Then the following are equivalent
\begin{enumerate}[(a)]
\item $H$ and $H^*$ together with $S$ satisfy the $\kappa$-entropy convexity inequality for all $x \in E^\circ$.
\item For any entropic interpolation $\{x(t)\}_{0 \leq t \leq T}$  satisfying Assumption \ref{assumption:entropic_interpolations_not_on_boundary}, we have for all $t \in [0,T]$ that
\begin{multline*}
S(x(t)) \leq \frac{T-t}{T} S(x(0)) + \frac{t}{T} S(x(T)) \\
- \kappa \int_0^T \left[\cL(x(s),\dot{x}(s)) + \cL^*(x^*(T-s),\dot{x}^*(T-s)) \right] G_T(s,t) \dd s.
\end{multline*}
\item For any entropic interpolation $\{x(t)\}_{0 \leq t \leq T}$  satisfying Assumption \ref{assumption:entropic_interpolations_not_on_boundary}, we have
\begin{multline*}
\frac{\dd}{\dd t}|_{t = 0} S(x(t)) \leq \frac{1}{T}\left[S(x(T)) - S(x(0)) \right] \\
- \kappa \int_0^T \left[\cL(x(s),\dot{x}(s)) + \cL^*(x^*(T-s),\dot{x}^*(T-s)) \right] \frac{T-s}{T} \dd s
\end{multline*}
and
\begin{multline*}
-\frac{\dd}{\dd t}|_{t = T} S(x(t)) \leq \frac{1}{T}\left[S(x(0)) - S(x(T)) \right] \\
- \kappa \int_0^T \left[\cL(x(s),\dot{x}(s)) + \cL^*(x^*(T-s),\dot{x}^*(T-s)) \right] \frac{s}{T} \dd s.
\end{multline*}
\item For any entropic interpolation $\{x(t)\}_{0 \leq t \leq T}$  satisfying Assumption \ref{assumption:entropic_interpolations_not_on_boundary}, we have
\begin{multline*}
\frac{\dd}{\dd t}|_{t = T} S(x(t)) - \frac{\dd}{\dd t}|_{t = 0} S(x(t))  \\
\geq \kappa \int_0^T \left[\cL(x(s),\dot{x}(s)) + \cL^*(x^*(T-s),\dot{x}^*(T-s)) \right] \dd s.
\end{multline*}
\end{enumerate}
In particular, if $\kappa \geq 0$, we have convexity of the entropy along entropic interpolations satisfying Assumption \ref{assumption:entropic_interpolations_not_on_boundary}.
\end{theorem}

\begin{proof}
That (a) implies (b) follows as noted above. The other implications follow from elementary computations, see for example the proof of Proposition 16.2 in \cite{Vi09}.
\end{proof}

\begin{remark}
In the context of a reversible system and where $H(x,p)$ is quadratic in $p$, an entropic interpolation has constant `speed'. Thus, the integrals on the right hand sides of (b)-(d) can be transformed into distances. Thus, (b) extends the $\kappa$-convexity of the entropy along displacement interpolations from the setting where $H$ is purely quadratic(no linear term):
\begin{equation*}
S(x(t)) \leq \frac{T-t}{T} S(x(0)) + \frac{t}{T} S(x(T)) \\
- \kappa \frac{t(T-t)}{2 T^2} d^2(x(0),x(T)).
\end{equation*}
In the measure valued setting, compare this to the Benamou-Brenier formula in optimal transport where the associated Hamiltonian is also purely quadratic.

In this light, note that even though this is not immediately possible in (b) and (c), we can replace the integral in the right-hand side of (d) by a time-dependent distance-like object based on the Lagrangian $\overline{\cL}$ by Proposition \ref{proposition:equivalence_entropic_interpolations}.
\end{remark}

\begin{remark}
Together with the connection made by \cite{MPR13} between large deviations and gradient flows for the entropy, the introduction of path-space large deviations into the problem of exponential entropy decay and convexity of the entropy along interpolations seems to give partial answers to the open questions 6.1 (b) and (c) in \cite{Le13}.
\end{remark}

The following lemma connects the entropy convexity inequalities with the entropy-information inequality in the case that $H = H^*$.

\begin{lemma}
Let $H$ and $S$ satisfy Assumption \ref{assumption:basic_assumption}. If $H$ satisfies the $\kappa$-entropy convexity inequality with $\kappa > 0$, then $H$ satisfies inequality \eqref{eqn:second_derivative_inequality_McKean_Vlasov} in Proposition \ref{proposition:EII_entropy_decay} and if there is only one stationary point $x_s$ for the McKean-Vlasov equation where $S(x_s) = 0$ then EII($\kappa$) is satisfied.
\end{lemma}

\begin{proof}
Suppose $\{x(t)\}_{t \geq 0}$ is a solution to the McKean-Vlasov equation. Then $\cL(x(s),\dot{x}(s)) = 0$ and by Lemma \ref{lemma:derivative_of_entropy} $\cL^*(x^*(T-s),\dot{x}^*(T-s)) = - \frac{\dd}{\dd s} S(x(s))$.

Thus, by (c) of Theorem \ref{theorem:convexity_of_entropy_along_interpolations}, we find for all $T > 0$ that
\begin{equation*}
\frac{\dd}{\dd t}|_{t = T} S(x(T)) - \frac{\dd}{\dd t}|_{t = 0} S(x(0))  \\
\geq - \kappa \int_0^T \frac{\dd}{\dd s} S(x(s)) \dd s.
\end{equation*}
This implies that
\begin{equation*}
\frac{\dd^2}{\dd t^2} S(x(t)) \geq - \kappa \frac{\dd}{\dd t} S(x(t)),
\end{equation*}
which is a reformulation of \eqref{eqn:second_derivative_inequality_McKean_Vlasov}.
\end{proof}

\section{Tensorization} \label{section:tensorization}

In this section, we will consider two variants of tensorization and show that the inequalities introduced above behave well under tensorization.
One variant corresponds intuitively to low-noise systems, and one to the large deviations of empirical densities. For both, we formally motivate the construction by going back to \eqref{eqn:LDP}.

\smallskip

Suppose we have time-homogeneous Markov processes $X^i_n(t)$ on $F_i$ for $i \in \{1,\dots, k\}$ that have generators $A_{i,n}$. The approach to prove path-space large deviations by Feng and Kurtz, \cite{FK06}, also applied in \cite{DFL11} and \cite{Kr16b,CoKr16}, shows that the large deviation principle as $n \rightarrow \infty$ for the sequences $X_{i,n} \in D_{F_i}(\bR^+)$ can formally be obtained in the following way.
\begin{enumerate}[(a)]
\item Define the operators $H_{i,n}f := \frac{1}{n} e^{-nf} A_{n,i} e^{nf}$.
\item  Show that for a sufficiently large class of functions there exists a limiting operator $H_{n} f$ such that $\lim_n \vn{H_{i,n}f - H_i f} = 0$.
\item Show that $H_i$ is of the form $H_i f(x) = H_i(x,\nabla f(x))$.
\item Define the Lagrangian $\cL_i(x,v) = \sup_p \ip{p}{v} - H_i(x,p)$.
\end{enumerate}
If the sequence $X_{i,n}(0)$ satisfies the large deviation principle with good rate function $I_{i,0}$, then the rate function of $X_{i,n}$ formally equals
\begin{equation*}
\PR\left[X_{n,i} \approx \gamma\right] \approx e^{-nI(\gamma)},
\end{equation*}
where $I(\gamma) = I_0(\gamma(0)) + \int_0^\infty \cL(\gamma(t),\dot{\gamma}(t)) \dd t$ for absolutely continuous $\gamma$ and $I(\gamma) = \infty$ otherwise.

\smallskip

The process $X_n(t) := (X^1_n(t), \dots, X^k_n(t))$ on $F := \prod_{i=1}^k F_i$ is also Markovian. Denote by $\cD(A_n)$ the linear span of functions of the type $f(x) = f(x_1,\dots, x_k) = \prod_{i=1}^k f_i(x_i)$, where $f_i \in \cD(A_{i,n})$. Then, the generator $A_n$ of the process $X_n(t)$ is given for $f \in \cD(A_n)$, by
\begin{equation*}
A_nf(x) = \sum_{i=1}^k \left(\prod_{j \neq i} f_j(x_j) \right) A_if_i(x_i).
\end{equation*}
For the domain of $H_n$, we consider functions of the form $f$ such that $e^{nf} \in \cD(A_n)$. In other words, $f$ of the form
\begin{equation*}
f(x) = \sum_{i=1}^k f_i(x_i)
\end{equation*}
for $f_i \in \cD(A_i)$. It follows that for $f$ of this type
\begin{equation*}
H_{n}f(x) = \sum_{i=1}^k \frac{1}{n} e^{-nf_i(x_i)} (A_{n,i} e^{nf_i})(x_i) = \sum_{i=1}^k H_{n,i}f_i(x_i).
\end{equation*}
We conclude that for $f$ of the form $f(x) = \sum_{i=1}^k f_i(x_i)$ the formal limit $Hf = \lim_n H_nf$ is of the form $Hf(x) = \sum_{i=1}^k H_i f_i(x_i)$. Writing this in terms of the gradient of $f$: $Hf(x) = H(x,\nabla f(x))$, we find that $H(x,p) = \sum_{i=1}^k H_i(x_i,p_i)$.

This formal computation leads to the following tensorization procedure.

\subsection{Tensorization for product systems}

For $i \in \{1,\dots,k\}$ let $E_i$ be a closed subset of $\bR^{d_k}$. Additionally, suppose that $\PR_{i,n}$ are measures on $D_{E_i}(\bR^+)$. Set $d := \sum_i d_i$ and denote $E = \prod_{i=1}^k E_i \subseteq \bR^d$. The product measures $\PR_n = \otimes_{i=1}^k \PR_{i,n}$ are defined on $\prod_{i=1}^k D_{E_i}(\bR^+)$ or equivalently on $D_E(\bR^+)$.

As we are taking a product system, it follows that if the trajectories under the measures $\PR_{i,n}$ satisfy the large deviation principle on $D_{E_i}(\bR^+)$ with
\begin{enumerate}[(a)]
\item Lagrangians $\cL_i$,
\item Hamiltonians $H_i$ that satisfy Assumptions \ref{assumption:basic_assumption} H(a) and H(b),
\item functions $S_i$ that satisfy \ref{assumption:basic_assumption} S(a) and S(b),
\end{enumerate}  
then the trajectories under $\PR_n$ satisfy a large deviation principle with Lagrangian $\cL(x,v) = \sum_{i=1}^k \cL_i(x_i,v_i)$, Hamiltonian $H(x,p) = \sum_{i=1}^k H_i(x_i,p_i)$ and entropy $S(x) = \sum_{i=1}^k S_i(x_i)$ that satisfy Assumption \ref{assumption:basic_assumption}, where $x = (x_1, \dots, x_k)$,$v = (v_1, \dots, v_k)$ and $p = (p_1, \dots, p_k)$.

The following proposition is straightforward, and follows from the principle that large deviation principles turn products into sums.

\begin{proposition} \label{proposition:tensorization_for_product_systems}
Suppose that for every $i \in \{1,\dots,k\}$ that $H_i, S_i$ satisfy Assumption \ref{assumption:basic_assumption}. Consider the product system with Hamiltonian $H$ and entropy $S$. Then we have the following implications.

\begin{enumerate}[(a)]
\item Suppose that for each $i$ $H_i$ and $S_i$ satisfy EII($\kappa_i$). Then $H$ and $S$ satisfy an entropy-information inequality with constant $\kappa := \min_{i} \kappa_i$.
\item Suppose that for each $i$ $H_i$ and $S_i$ satisfy \eqref{eqn:second_derivative_inequality_McKean_Vlasov} with constant $\kappa_i$. Then $H$ and $S$ satisfy \eqref{eqn:second_derivative_inequality_McKean_Vlasov}  with constant $\kappa := \min_{i} \kappa_i$.
\item Suppose that for each $i$ $H_i$ and $S_i$ satisfy ECI($\kappa_i$). Then $H$ and $S$ satisfy an entropy-convexity inequality with constant $\kappa := \min_{i} \kappa_i$.
\end{enumerate}
\end{proposition}

Additionally, we have the following technical result, which is useful for the application of Theorem \ref{theorem:convexity_of_entropy_along_interpolations}.

\begin{proposition} \label{proposition:tensorisation_stay_in_interior_first_variant}
Suppose that for every $i \in \{1,\dots,k\}$ that $H_i, S_i$ satisfy Assumption \ref{assumption:basic_assumption}. Suppose that for every $i$, we have that every entropic interpolation $\gamma : [0,T] \rightarrow E_i$ satisfies $\gamma(t) \in E_i^\circ$ for $t \in (0,T)$. Then it holds for every entropic interpolation $\rho : [0,T] \rightarrow E$ of the product system that $\rho(t) \in E^\circ$ for all $t \in (0,T)$.
\end{proposition}

\subsection{Tensorization for the evolution of empirical densities on product spaces} \label{section:tensorization_for_empirical_density_on_product_spaces}

Next, we consider Markov jump processes on a finite state-space. We will show that if we consider the dynamics of empirical averages that take their values in spaces of the type $E = \cP(\{1,\dots,q\})$ the inequalities introduced above also behave well under tensorization. In this setting, we can not immediately follow the formal argument as at the start of Section \ref{section:tensorization}, because the large deviation principle does not hold for the processes themselves, but for a lower-dimensional projection of the process.

To clarify what we mean by tensorization in this context, we go back to the underlying processes of equation \eqref{eqn:LDP}. Let $F_1, \dots, F_k$ be finite sets of sizes $d_i$ and set $F = \prod_{i=1}^k F_i$. Additionally, denote $E_i := \cP(F_i) \subseteq \bR^{d_i}$ and $E := \cP(F) \subseteq \bR^d$, where $d = \sum_{i=1}^k d_i$. 

Fix some $i$ and suppose that on $F_i$ we have $n$ interacting jump-processes $(X_1^i(t),\dots,X_n^i(t))$ such that the trajectories $t \mapsto \mu_n^i(t)$ of the empirical density
\begin{equation*}
\mu_n^i(t) := \frac{1}{n} \sum_{j=1}^n \delta_{X_j^n(t)} \in \cP(F_i) = E_i
\end{equation*}
satisfies the large deviation principle as in \eqref{eqn:LDP} with Lagrangian $\cL_i$, Hamiltonian $H_i$ and entropy $S_i$ satisfying Assumption \ref{assumption:basic_assumption}.

\smallskip

Now we consider product dynamics on the level of the jump processes on $F_i$. Thus, for each $n$, we have $n$ interacting jump processes $(Y_1(t),\dots,Y_n(t))$ on $F$, where $Y_j(t) := (X_j^1(t),\dots,X_j^k(t))$.

As we consider product dynamics, it follows that the trajectories $t \mapsto \mu_n(t)$, defined by
\begin{equation*}
\mu_n(t) := \frac{1}{n} \sum_{j=1}^n \delta_{Y_j} \in \cP(F) = E
\end{equation*}
also satisfies the large deviation principle on $D_E(\bR^+)$. To express the Lagrangian, Hamiltonian and entropy of this product system in terms of the ones corresponding to the separate systems, we introduce some notation.

\smallskip

Let $\pi_i : F \rightarrow F_i$ denote the projection map $\pi_i(a_1, \dots, a_k) = a_i$. $\pi$ induces the following maps:
\begin{enumerate}[(a)]
\item $\pi_i : E = \cP(F) \rightarrow E_i = \cP(F_i)$, by $\pi_i \mu = \mu \circ \pi_i^{-1}$.
\item $\pi_i : \bR^{|F|} \rightarrow \bR^{|F_i|}$ by $\pi_i(p) := \left\{(\pi_i(p))_j\right\}_{j \in F_i}$ where
\begin{equation*}
(\pi_i(p))_j := \sum_{\substack{q \in F \\ \pi_i(q) = j}} p_q.
\end{equation*}
\end{enumerate}

Arguing as at the start of Section  \ref{section:tensorization}, it follows that the Lagrangian, Hamiltonian and entropy for the large deviation principle for the product system are given by
\begin{enumerate}[(a)]
\item $\cL(x,v) = \sum_{i=1}^k \cL_i(\pi_i(x),\pi_i(v))$,
\item $H(x,p) = \sum_{i=1}^k H_i(\pi_i(x),\pi_i(p))$,
\item $S(x) =  \sum_{i=1}^k S_i(\pi_i(x))$.
\end{enumerate}
As before, if the components satisfy Assumption \ref{assumption:basic_assumption}, then this Assumption is also satisfied for the product system. We have the following two analogous results.

\begin{proposition} \label{proposition:tensorization_for_empirical_measure_product}
Suppose that for every $i \in \{1,\dots,k\}$ that $H_i, S_i$ satisfy Assumption \ref{assumption:basic_assumption}. Consider the product system with Hamiltonian $H$ and entropy $S$. Then we have the following implications.

\begin{enumerate}[(a)]
\item Suppose that for each $i$ $H_i$ and $S_i$ satisfy an EII($\kappa_i$). Then $H$ and $S$ satisfy an entropy-information inequality with constant $\kappa := \min_{i} \kappa_i$.
\item Suppose that for each $i$ $H_i$ and $S_i$ satisfy \eqref{eqn:second_derivative_inequality_McKean_Vlasov} with constant $\kappa_i$. Then $H$ and $S$ satisfy \eqref{eqn:second_derivative_inequality_McKean_Vlasov}  with constant $\kappa := \min_{i} \kappa_i$.
\item Suppose that for each $i$ $H_i$ and $S_i$ satisfy EII($\kappa_i$). Then $H$ and $S$ satisfy an entropy-convexity inequality with constant $\kappa := \min_{i} \kappa_i$.
\end{enumerate}
\end{proposition}

Additionally, we have the following variant of Proposition \ref{proposition:tensorisation_stay_in_interior_first_variant}.

\begin{proposition} \label{proposition:tensorization_for_empirical_measure_product_trajectories_in_interior}
Suppose that for every $i \in \{1,\dots,k\}$ that $H_i, S_i$ satisfy Assumption \ref{assumption:basic_assumption}. Suppose that for every $i$, we have that every entropic interpolation $\gamma : [0,T] \rightarrow E_i$ satisfies $\gamma(t) \in E_i^\circ$ for $t \in (0,T)$. Then it holds for every entropic interpolation $\rho : [0,T] \rightarrow E$ of the product system that $\rho(t) \in E^\circ$ for all $t \in (0,T)$.
\end{proposition}

\section{Examples} \label{section:examples}

We verify the various inequalities introduced above for the limiting dynamics of five examples in increasing order of complexity:
\begin{enumerate}[(a)]
\item The generalized Ornstein-Uhlenbeck processes with vanishing diffusion constant. We give conditions for the $\kappa$-entropy-convexity inequality with optimal constant. The constant coincides with the optimal lower bound for the Ricci-curvature, cf. Section 1.16 in \cite{BaGeLe14}.
\item The Kramers equation, or the underdamped Langevin equation with vanishing diffusion constant. The interesting feature is the non-reversibility of this system. We prove the $\kappa$-entropy-information inequality with optimal constant and show that although entropy decreases exponentially fast, information does not.
\item The Wright-Fisher model for $d$ species with parent independent mutation rates with vanishing diffusion constant. The large deviation principle gives a `non-standard' entropy. We establish the exponential decay of information in the setting that $d \geq 2$ and the $\kappa$-entropy-convexity inequality with optimal constant for $d = 2$.
\item The empirical magnetization for Glauber dynamics with inverse temperature $\beta \geq 0$ on the Curie-Weiss model. If $\beta \leq 1$, we establish a optimal $4(1-\beta)$-entropy-convexity bound.
\item The empirical law of mean-field interacting random walks on a hypercube $\{-1,1\}^N$, for which we establish the $\kappa$-entropy-convexity bound. In the non-interacting case this constant is $4/N$.
\end{enumerate}

\subsection{The generalized Ornstein-Uhlenbeck process}

An example where we can easily verify an $\kappa$-entropy-convexity inequality is for the Hamiltonian corresponding to the generalized Ornstein-Uhlenbeck process. Let $V : \bR^d \rightarrow [0,\infty)$ be some twice continuously differentiable convex function. Consider the following sequence of processes:
\begin{equation*}
\dd X_n(t) = - \nabla V(X_n(t)) \dd t + \frac{1}{ \sqrt{n}} \dd W(t).
\end{equation*}
The Freidlin-Wentzell large deviation principle of the trajectories of these processes gives an operator 
\begin{equation*}
H(x,p) = \frac{1}{2}\sum_i p_i^2 - p_i V_i(x),
\end{equation*}
where $V_i$ is the derivative of $V$ in the $i$-th coordinate. The associated entropy $S$ is given by $S(x) = 2V(x)$

\begin{theorem} 
Consider $H$ and $S$ introduced above. Then we have the entropy-convexity inequality with the largest constant $\kappa \in \bR$ such that the matrix
\begin{equation*}
\nabla \nabla V - \kappa \bONE
\end{equation*}
is non-negative definite. Consequently, the conclusions of Theorem \ref{theorem:convexity_of_entropy_along_interpolations} hold for the entropy $S(x) = 2 V(x)$. 
\end{theorem}

Clearly, in this setting Assumption \ref{assumption:entropic_interpolations_not_on_boundary} is satisfied. Thus this result holds for all entropic interpolations.

Note that this constant corresponds with the optimal lower bound on the classical Ricci-curvature, see Section 1.16 in \cite{BaGeLe14}.

\begin{proof}
It is immediate to verify that $H = H^*$, so we only check the entropy-convexity inequality for $H$. On one hand, we have $pH_p(x,p) - H(x,p) = \frac{1}{2} \sum_i p_i^2$, whereas on the other
\begin{multline*}
p H_{px}(x,p)H_p(x,p) - p H_{pp}(x,p)H_x(x,p) - H_p(x,p)H_x(x,p)  \\
= \sum_{i,j} p_i V_{i,j}(x) p_j.
\end{multline*}
\end{proof}

\subsection{The underdamped Langevin equation} \label{section:exp_decay_Langevin}

Next, we consider the empirical average of trajectories of particles and their momenta $(X_n(t),\rho_n(t)) \in \bR^{2d}$ evolving according to the underdamped Langevin equation with mass $m$, in a twice continuously differentiable potential $V : \bR^d \rightarrow \bR$, $V(x) = \sum_{i=1}^d V_i(x_i)$ given by
\begin{equation*}
\begin{bmatrix}
\dd X_n(t) \\ \dd \rho_n(t) 
\end{bmatrix}
= 
\begin{bmatrix} \frac{\rho_n(t)}{m} \dd t \\ - \nabla_x V(X_n(t)) \dd t - \gamma \frac{\rho_n(t)}{m} \dd t + \sqrt{\frac{2\gamma \theta}{n}} \dd W(t)\end{bmatrix},
\end{equation*}
where $\gamma, \theta > 0$ are two constants with physical interpretation, see \cite{DPZ13}, and where $\nabla_x$ is the gradient of $V$ in the position coordinates and where $W$ is a standard $d$-dimensional Brownian motion. Note that in contrast to the models considered above, the underdamped-Langevin dynamics are non-reversible. 

\smallskip

Sending $n$ to infinity, we have the Freidlin-Wentzell large deviation principle with Hamiltonian
\begin{equation*}
H\left(\begin{pmatrix}
x \\ \rho 
\end{pmatrix}, \begin{pmatrix}
p_x \\ p_\rho 
\end{pmatrix}\right) = \sum_{i=1}^d \frac{\rho_i}{m}p_{x,i} - \nabla_{x_i} V(x)p_{\rho_i} - \gamma \frac{\rho_i}{m}p_{\rho_i} + \gamma \theta |p_\rho|^2.
\end{equation*}
The momentum $\rho$ should not be confused with the second kind of momentum, i.e. the variables in the second input for $H$: the vector $(p_x,p_\rho)$.

If $V$ is growing sufficiently fast for $|x|$ large, the stationary measures $\mu_n$ of the dynamics are given by
\begin{equation*}
\mu_n(\dd x \, \dd \rho) = \frac{1}{Z_n} e^{-n\theta^{-1}(V(x) +  \frac{|\rho|^2}{2m})} \dd x \, \dd \rho,
\end{equation*}
where $Z_n$ is an appropriate normalising constant. This motivates the use of the entropy
\begin{equation*}
S(x,\rho) = \theta^{-1} V(x) + \frac{|\rho|^2}{2\theta m}.
\end{equation*}

The next result gives us the entropy-information inequality. Additionally, it shows that the stronger inequality \eqref{eqn:second_derivative_inequality_McKean_Vlasov} is not satisfied globally. 

\begin{proposition} 
Consider $H$ and $S$ corresponding to the underdamped Langevin equation. Then we have the $\frac{2\gamma}{m}$ entropy-information inequality, and the constant is optimal. Additionally, we have
\begin{multline} \label{eqn:langevin_entropy_convexity}
\ip{DS(x)}{H_{px}(x,0) H_p(x,0)} + \ip{D^2 S(x)H_p(x,0)}{H_p(x,0)} - \kappa I(x) \\
= \left(2\frac{\gamma}{m} - \kappa\right)\sum_{i=1}^d  \frac{\gamma }{\theta m^2}\rho_i^2 + \frac{\gamma}{\theta m^2} \rho_i \nabla V_i(x_i).
\end{multline}

Suppose we have $\nabla_{x_i} V(0) \geq 0$ for all $i$. Then $\Omega := \{(x,\rho) \in \bR^{2d} \, | \, \forall , i: \, x_i\rho_i \geq 0\}$ is closed under the dynamics of the McKean-Vlasov equation: i.e. if $(x(t),\rho(t))$ to $(\dot{x},\dot{\rho}) = H_{(x,\rho)}((x,\rho),0)$ with $(x(0),\rho(0)) \in \Omega_\beta$. Then $(x(t),\rho(t)) \in \Omega$ for all $t \geq 0$.
For such trajectories, we have exponential decay of $I$ with speed $\frac{2\gamma}{m}$.
\end{proposition}

\begin{remark}
For discussion of a setting where entropy decays exponentially, but the information does not, see Section 4.2 in \cite{CDP09}. For an example in a continuous setting see the example by Helffer following Proposition 1.5 in \cite{Le01}.
\end{remark}

\begin{proof}
The calculations in this setting are tedious but straightforward, we only do this for the setting $d = 1$, so that the full result follows by Proposition \ref{proposition:tensorization_for_product_systems}. We give the formula's for the main quantities:
\begin{equation*}
DS(x,\rho) =  \begin{bmatrix}
\theta^{-1} \nabla_x V(x) \\ \frac{\theta^{-1} \rho}{m}  
\end{bmatrix}, \qquad 
H_{(p_x,p_\rho)}((x,\rho),0) = \begin{bmatrix}
\rho/m \\ -\nabla_x V(x) - \frac{\gamma \rho}{m}  
\end{bmatrix}.
\end{equation*}
We conclude that $I(x,\rho) = \frac{\gamma \rho^2}{\theta m^2}$. This immediately yields the $\frac{2\gamma}{m}$ entropy-information inequality. Because $\inf_x V(x) = 0$ the constant is optimal.

Additionally, we have
\begin{align*}
D^2S(x,\rho) & =  \begin{bmatrix}
\theta^{-1} \nabla_x \nabla_x V(x) & 0 \\ 0 & \frac{1}{\theta m}  
\end{bmatrix}, \\
H_{(p_x,p_\rho),(x,\rho)}((x,\rho),0) & = \begin{bmatrix}
0 & 1/m \\ -\nabla_x\nabla_x V(x) & - \frac{\gamma }{m}  
\end{bmatrix}.
\end{align*}
Carrying out all multiplications yields \eqref{eqn:langevin_entropy_convexity}. Closedness of $\Omega$ under the dynamics and exponential decay of $I$ follows from Proposition \ref{proposition:EII_entropy_decay}
\end{proof}

For specific potentials $V$, we can extend our analysis. The next proposition is in the setting where $d=1$ and $V$ is quadratic, the proof of which is straightforward.

\begin{proposition}
Suppose $V(x) = \frac{1}{2}x^2$. Pick $\beta \in [0,\gamma / m]$ and set $\Omega_\beta := \{(x,\rho) \in \bR^2 \, | \, \rho(\rho\beta +x) \geq 0\}$. Then $\Omega_\beta$ is closed under the dynamics of the McKean-Vlasov equation: i.e. if $(x(t),\rho(t))$ to $(\dot{x},\dot{\rho}) = H_{(x,\rho)}((x,\rho),0)$ with $(x(0),\rho(0)) \in \Omega_\beta$. Then $(x(t),\rho(t)) \in \Omega_\beta$ for all $t \geq 0$. Additionally, we have
\begin{equation*}
S(x(t),\rho(t)) \leq e^{-\kappa t}S(x(0),\rho(0)),
\end{equation*}
for $\kappa = \frac{2\gamma}{m} - \beta$.
\end{proposition}

\subsection{The Wright-Fisher model} \label{section:Wright_Fisher}

Set $E = \cP(\{1,\dots,d\}) = \{x \in \bR^d \, | \, x_i \geq 0, \sum x_i = 1\}$. The large deviations of the trajectories of the Wright-Fisher model are considered in \cite{DaFe98}, and the Hamiltonian corresponding to this LDP is given by
\begin{equation*}
H(x,p) = \frac{1}{2} \sum_{i,j} x_i(\delta_{ij} - x_j) p_i p_j + \sum_{i = 1}^d\left(\sum_{j \neq i} x_j q_{ji} - x_i q_{ij} \right) p_i,
\end{equation*}
where $q_{ji}$ represents the mutation rate from $j$ to $i$.
In the case that the mutation rates are parent independent: $q_{ji} = \frac{1}{2}\mu_i > 0$, for $i \neq j$, the stationary measures of the associated Wright-Fisher processes with vanishing diffusion coefficient have entropy $S(x)$ given by
\begin{equation*}
S(x) =  \sum_{i=1}^d \mu_i \log \frac{\mu_i}{\mu x_i}.
\end{equation*}

For parent independent mutation rates, we have the following non-optimal result.

\begin{proposition} \label{proposition:EII_for_Wright_Fisher}
Consider $H$ and $S$ as above in the setting that $q_{ji} = \frac{1}{2}\mu_i > 0$ for all $i,j$. Define $\mu = \sum_i \mu_i$. Then $H$ and $S$ satisfy \eqref{eqn:second_derivative_inequality_McKean_Vlasov} and the entropy-information inequality with constant $\frac{1}{2} \mu $.
\end{proposition}

In the proof below, verify \eqref{eqn:reduction:EII_for_convex_S} only, which is sub-optimal. For the $d=2$ case, we show in Theorem \ref{theorem:WF_two_species_entropy_convexity} below that we can improve upon this constant and extend it to an entropy-convexity inequality.

\begin{proof} 
For the verification of \eqref{eqn:second_derivative_inequality_McKean_Vlasov} with constant $\frac{1}{2} \mu $ we observe that $x \mapsto S(x)$ is convex, so it suffices to verify \eqref{eqn:reduction:EII_for_convex_S}. Thus, we calculate the vector $H_p(x,0)$ and matrix $H_{px}(x,0)$. We find
\begin{multline*}
H_{p_j}(x,0) = \frac{1}{2} \sum_{l \neq j} x_l \mu_j - x_j \mu_l \\
= \frac{1}{2}\left((1-x_j)\mu_j - x_j (\mu-\mu_j) \right) = \frac{1}{2}\left(\mu_j - x_j \mu \right),
\end{multline*}
\begin{equation*}
H_{p_i, x_j}(x,0) = \begin{cases}
0 & \text{for } i \neq j, \\
- \frac{1}{2} \mu & \text{for } i = j.
\end{cases} 
\end{equation*}
We find that
\begin{equation*}
\sum_j H_{p_i,x_j}(x,0) H_{p_j}(x,0) = - \frac{1}{2} \mu H_{p_i}(x,0),
\end{equation*}
and as a consequence
\begin{equation*}
\ip{DS(x)}{H_{px}(x,0)H_p(x,0)} = - \frac{1}{2} \mu \ip{DS(x)}{H_p(x,0)} = \frac{1}{2}\mu I(x).
\end{equation*}
\end{proof}

Before proving the entropy-convexity inequality in the $d=2$ case, we verify Assumption \ref{assumption:entropic_interpolations_not_on_boundary}. First, note that we can re-express the model in terms of $x = x_2 \in [0,1]$, so that $(x_1,x_2) = (1-x,x)$. In the variable $x$, the Hamiltonian is given by
\begin{equation*}
H(x,p) = \frac{1}{2}a(x)p^2 - b(x)p, \quad a(x) = x(1-x), \quad b(x) = \frac{1}{2}\left(x \mu_1 - (1-x)\mu_2\right)
\end{equation*}
and the entropy $S$ reduces in this setting to
\begin{equation*}
S(x) = \mu_1 \log \frac{\mu_1}{\mu(1-x)} + \mu_2 \log \frac{\mu_2}{\mu x}.
\end{equation*}

\begin{lemma} \label{lemma:WF_cost_for_stationarity_is_smaller_in_interior}
Let $\mu_1, \mu_2 > 0$ and $\mu = \mu_1 + \mu_2$. Then assumption \ref{assumption:entropic_interpolations_not_on_boundary} is satisfied: 
\begin{enumerate}[(a)]
\item $H = H^*$,
\item $H_p(x,0) > 0$ and $H_p(x,1) < 0$,
\item the maps $x \mapsto \cL(x,0)$ and $x \mapsto S(x)$ are decreasing on an open neighbourhood $U_{0}$ of $0$ and increasing on an open neighbourhood $U_1$ of $1$,
\item we have
\begin{align*}
& \lim_{x \downarrow 0} \frac{\cL(x,0) - \cL(0,0)}{S(x) - S(0)} = \infty \\
& \lim_{x \uparrow 1} \frac{\cL(x,0) - \cL(1,0)}{S(x) - S(1)} = \infty.
\end{align*}
\end{enumerate}
As a consequence, Assumption \ref{assumption:entropic_interpolations_not_on_boundary} is satisfied for this model.
\end{lemma}

\begin{proof}
Using that $DS(x) = \frac{2b(x)}{a(x)}$, it is straightforward to verify that $H = H^*$. We have $H_p(x,0) = -b(x)$, so that $H_p(0,0) > 0$ and $H_p(1,0) < 0$ by the positivity of $\mu_1$ and $\mu_2$.

As $DS(x) = \frac{2b(x)}{a(x)}$, we find that
\begin{equation*}
\cL(x,0) = - H(x,\frac{1}{2}DS(x)) = \frac{1}{2} \frac{b(x)^2}{a(x)}.
\end{equation*}
Differentiating this with respect to $x$ yields
\begin{equation*}
2 \frac{\dd}{\dd x} \cL(x,0) = \frac{2 a(x) b(x)b'(x) - a'(x)b(x)^2}{a(x)^2}. 
\end{equation*}
To verify the third claim for $\cL$, we need to know the sign of this derivative. As the denominator is non-negative, we calculate the numerator(recall that $\mu = \mu_1 + \mu_2)$:
\begin{equation*}
2 a(x) b(x)b'(x) - a'(x)b(x)^2 = \frac{1}{4}\left[x \mu - \mu_2\right] \left[\left(\mu - 2 \mu_2\right) x + \mu \right].
\end{equation*}
Thus the claim in (c) for $\cL$ follows as this quantity is negative for $x$ close to $0$ and positive for $x$ close to $1$. The statement for $S$ is clear.

\smallskip

We verify (d) only for the left-hand boundary. The claim follows if we can show that $\frac{\dd}{\dd x} \cL(x,0)$ diverges to $- \infty$ faster than $DS(x)$ diverges to $-\infty$.

Note that
\begin{equation*}
2\frac{\dd}{\dd x} \cL(x,0) = DS(x) \frac{2a(x)b'(x) - a'(x)b(x)}{a(x)}
\end{equation*}
As $DS(x) < 0$ for $x$ close to $0$, we have to show that
\begin{equation*}
\frac{2a(x)b'(x) - a'(x)b(x)}{a(x)} = \frac{x \mu_1 - x \mu_2 + \mu_2}{2x(1-x)}
\end{equation*}
diverges to $\infty$ as $x \downarrow 0$. This, however, is immediate from the $\frac{1}{x}$ term in the denominator and the positive $\mu_2$ term in the numerator.
\end{proof}

In this one-dimensional setting, we improve the constant of the entropy-information inequality of Proposition \ref{proposition:EII_for_Wright_Fisher} and extend it to the entropy-convexity inequality. 

\begin{theorem}[Wright-Fisher model with positive mutation rates for two species] \label{theorem:WF_two_species_entropy_convexity}
Let $\mu_1, \mu_2 > 0$ and let $\mu = \mu_1 + \mu_2$. Let $H$ be the Hamiltonian given by
\begin{equation*}
H(x,p) = \frac{1}{2}a(x)p^2 - b(x)p, \quad a(x) = x(1-x), \quad b(x) = \frac{x \mu_1 - (1-x)\mu_2}{2},
\end{equation*}
and where $S$ is given by
\begin{equation*}
S(x) = \mu_1 \log \frac{\mu_1}{\mu(1-x)} + \mu_2 \log \frac{\mu_2}{\mu x}.
\end{equation*}
Then $H$ satisfies the entropy-convexity inequality and the conclusions of Theorem \ref{theorem:convexity_of_entropy_along_interpolations} with respect to $S$ with constant 
\begin{equation*}
\kappa = \frac{1}{2}\mu + \frac{1}{2}\sqrt{\mu^2 - (\mu_1 - \mu_2)^2} = \frac{1}{2}\mu + \sqrt{\mu_1 \mu_2}
\end{equation*}
for all entropic interpolations. Additionally, this constant is optimal for the entropy-convexity inequality.
\end{theorem}

\begin{proof}
To start, we find
\begin{equation*}
pH_p(x,p) - H(x,p) = \frac{1}{2} a(x) p^2.
\end{equation*}
A second tedious, but straightforward, calculation yields
\begin{multline*}
p H_{px}(x,p) H_p(x,p) - p H_{pp}(x,p)H_x(x,p) - H_x(x,p) H_p(x,p) \\
= \left[a(x)b'(x) - \frac{1}{2}a'(x) b(x)\right] p^2.
\end{multline*}
Using the definitions of $a$ and $b$, we conclude that we need to find the largest $\kappa$ for which 
\begin{equation*}
2\kappa (x - x^2) \leq (\mu_1 - \mu_2)x + \mu_2
\end{equation*}
is satisfied for all $x \in [-1,1]$. As $\mu_1, \mu_2 >0$, there is at least some $\kappa > 0$ for which this inequality is satisfied. To find the largest $\kappa > 0$ for which this is the case, the minimum of
\begin{equation*}
f_\kappa(x) := 2\kappa x^2 + (\mu_1 - \mu_2 - 2\kappa) x + \mu_2
\end{equation*}
for $x \in [0,1]$ should equal $0$. As $f_\kappa$ is convex for $\kappa > 0$, the derivative in $x$ of $f_\kappa$ is increasing. As $f_\kappa(0), f_\kappa(1) > 0$, $\kappa$ must be such that $f'_\kappa(0) < 0$ and $f'_\kappa(1) > 0$. We conclude that $2\kappa > |\mu_1 - \mu_2|$. The location of the minimum of $f_\kappa$ is found at
\begin{equation*}
x_{min}(\kappa) = \frac{1}{2} - \frac{\mu_1 - \mu_2}{4 \kappa}.
\end{equation*}
Evaluating the parabola in its minimum and putting this equal to $0$ gives an equation for the value of $\kappa$:
\begin{equation*}
(\mu_1 - \mu_2 - 2\kappa)^2 - 8 \kappa \mu_2 = 0
\end{equation*}
which is equivalent to solving
\begin{equation*}
4\kappa^2 - 4 \mu \kappa + (\mu_1 - \mu_2)^2 = 0.
\end{equation*}
Both zeros are non-negative, but an elementary computation shows that the smallest solution is smaller than $\frac{1}{2}|\mu_1 - \mu_2|$. We conclude that the largest suitable $\kappa$ equals 
\begin{equation*}
\kappa = \frac{1}{2}\mu + \frac{1}{2}\sqrt{\mu^2 - (\mu_1 - \mu_2)^2} = \frac{1}{2}\mu + \sqrt{\mu_1 \mu_2}.
\end{equation*}
We did not use any inequalities in the identification of $\kappa$, which implies that the constant is optimal.
\end{proof}

\subsection{Glauber dynamics for the Curie-Weiss model} \label{section:Glauber_for_CW}

The fourth example considers the limiting dynamics of the magnetization of the Curie-Weiss model evolving under Glauber dynamics with potential $V(x) = -\frac{1}{2}\beta x^2$, see for example 2.9 and 2.12 in \cite{Kr16b} or \cite{Co89,Le95}.

The Hamiltonian is given by 
$H : [-1,1] \times \bR \rightarrow \bR$ defined by
\begin{equation*}
H(x,p) = \frac{1-x}{2}e^{\beta x}\left[e^{2p} - 1\right] + \frac{1+x}{2}e^{-\beta x}\left[e^{-2p} - 1\right]
\end{equation*}
and the associated entropy functional is given by
\begin{equation*}
S(x) = \frac{1-x}{2} \log (1-x) + \frac{1+x}{2} \log (1+x) - \frac{1}{2}\beta x^2 + C_\beta,
\end{equation*}
where $C_\beta$ is such that $\inf_x S(x) = 0$.

We introduce two auxiliary functions that turn up in the analysis at various points. Define
\begin{align*}
G_1(x) & := \cosh(\beta x) - x \sinh(\beta x), \\
G_2(x) & := \sinh(\beta x) - x \cosh(\beta x),
\end{align*}
and note that the Hamiltonian can be rewritten in terms of $G_1$ and $G_2$ as
\begin{equation*}
H(x,p) = \left[\cosh(2p) - 1\right]G_1(x) + \sinh(2p) G_2(x).
\end{equation*}

The following lemma follows from the definitions of $G_1$ and $G_2$.
\begin{lemma} \label{lemma:properties_of_G1G2}
For $\beta \in [0,1]$, the functions $G_1,G_2 : [-1,1] \rightarrow \bR$ have the following properties:
\begin{enumerate}[(a)]
\item $G_1$ is even, positive, increasing for $x \leq 0$ and decreasing for $x \geq 0$,
\item $G_2$ is odd, positive for $x \leq 0$, negative for $x \geq 0$ and decreasing.
\end{enumerate}
\end{lemma}

We start out by verifying Assumption \ref{assumption:entropic_interpolations_not_on_boundary}.

\begin{lemma} \label{lemma:CW_cost_for_stationarity_is_smaller_in_interior}
Assumption \ref{assumption:entropic_interpolations_not_on_boundary} is satisfied: $H = H^*$, $H_p(-1,0) > 0$, $H_p(1,0) > 0$, the map $x \mapsto \cL(x,0)$ is decreasing on an open neighbourhood $U_{-1}$ of $-1$ and increasing on an open neighbourhood $U_1$ of $1$, and we have
\begin{align*}
& \lim_{x \downarrow -1} \frac{\cL(x,0) - \cL(-1,0)}{S(x) - S(-1)} = \infty \\
& \lim_{x \uparrow 1} \frac{\cL(x,0) - \cL(1,0)}{S(x) - S(1)} = \infty.
\end{align*}
As a consequence, Assumption \ref{assumption:entropic_interpolations_not_on_boundary} is satisfied for this model.
\end{lemma}

\begin{proof}
The first claim follows from a direct computation. For the second claim, note that $H_p(x,0) = 2G_2(x)$. By Lemma \ref{lemma:properties_of_G1G2}, we find $H_p(-1,0) = G_2(-1) > 0$ and $H_p(1,0) = 2G_2(1) > 0$.

The third claim is immediate from
\begin{equation*}
\cL(x,0) = - \inf_p H(x,p) = - H_p\left(x,\frac{1}{2}DS(x)\right) = - \sqrt{1-x^2} + G_1(x).
\end{equation*}
The square root has diverging derivative for $x$ close to the boundary, whereas the second term is continuously differentiable on $[-1,1]$, thus we obtain the result.

\smallskip

We only verify the fourth claim for the left boundary. In particular, it is sufficient to show that $\frac{\dd}{\dd x} \cL(x,0)$ diverges to $ - \infty$ faster that $DS(x)$ diverges to $- \infty$ as $x \downarrow -1$.
In particular, close to $-1$, we have
\begin{equation*}
\frac{\dd}{\dd x} \cL(x,0) =  -a(x)\frac{1}{\sqrt{1+x}} + c_1(x), \qquad DS(x) = \frac{1}{2} \log (1+x) + c_2(x),
\end{equation*}
where $c_1,c_2$ are functions that are bounded on a neighbourhood of $-1$ and where $a$ is a function close to $-1$ for $x$ close to $-1$. The result follows from the asymptotic behaviour of $\frac{-1}{\sqrt{1+x}}$ and $\log(1+x)$ close to $-1$.
\end{proof}

We conclude that in this setting entropic interpolations are in the interior except perhaps at the start and end-point. Our main theorem shows that the entropy-convexity inequality holds with a constant that nicely depends on $\beta$.

\begin{theorem}[Curie-Weiss jump process on two states] \label{theorem:interacting_jump_processes_on_two_states_entropy_convexity}
Consider the Hamiltonian $H : [-1,1] \times \bR \rightarrow \bR$ defined by
\begin{equation*}
H(x,p) = \left[\cosh(2p) - 1\right]G_1(x) + \sinh(2p) G_2(x)
\end{equation*}
for $\beta \leq 1$.
Then $H$ satisfies the entropy-convexity inequality with respect to the relative entropy $S(x) = \frac{1-x}{2} \log (1-x) + \frac{1+x}{2} \log (1+x) - \frac{1}{2}\beta x^2$ with constant $4(1-\beta)$ and thus the conclusions of Theorem \ref{theorem:convexity_of_entropy_along_interpolations} hold with constant $4(1-\beta)$ for all entropic interpolations. The constant is optimal for the entropy-information and the entropy-convexity inequalities.
\end{theorem}

As noted above $H = H^*$ in this case, so we only have to consider the $4(1-\beta)$ entropy-convexity inequality for $H$. The proof is based on the basic inequality that $(1-\beta)p \leq (1-\beta)\sinh(p)$ for $p \geq 0$, and thus does not immediately generalize for $\beta > 1$.

\begin{proof}[Proof of Theorem \ref{theorem:interacting_jump_processes_on_two_states_entropy_convexity}]
We will prove
\begin{multline*}
\kappa \left[\ip{p}{H_p(x,p)} - H(x,p)\right] \leq \ip{p}{H_{px}(x,p)H_p(x,p)} \\
- \ip{p}{H_{pp}(x,p)H_x(x,p)} - \ip{H_x(x,p)}{H_p(x,p)},
\end{multline*}
for $\kappa = 4(1-\beta)$. Note that for $p = 0$ all terms equal $0$. As the state-space for $p$ is one-dimensional and the problem is symmetric under flipping $(x,p)$ to $(-x,-p)$, it suffices to prove that the derivatives in $p$, for $p \geq 0$ for every fixed $x$ of the functions on the left and right hand side are ordered with the same constant $\kappa = 4(1-\beta)$:
\begin{multline*}
\kappa p H_{pp}(x,p) \leq p H_{pxp}(x,p)H_p(x,p) \\
- pH_{ppp}(x,p) H_x(x,p) - 2 H_{pp}(x,p) H_x(x,p).
\end{multline*}
Our argument will be based on the basic inequality that $2p \leq \sinh(2p)$ for $p \geq 0$. In particular, as $H_{pp}(x,p) > 0$ by the strict convexity of $H$ in the momentum variable this implies that 
\begin{equation*}
4(1-\beta)p H_{pp}(x,p) \leq 2 (1-\beta) \sinh(2p) H_{pp}(x,p).
\end{equation*}
Thus, it suffices to prove for $p \geq 0$ and all $x$ that
\begin{multline} \label{eqn:entropy_convexity_CW_final_to_prove}
0 \leq p H_{pxp}(x,p)H_p(x,p) - pH_{ppp}(x,p) H_x(x,p) \\
- 2 H_{pp}(x,p) \left(H_x(x,p) + (1-\beta)\sinh(2p) \right).
\end{multline}

To do this, we study the Hamiltonian in terms of $G_1$ and $G_2$ as
\begin{equation*}
H(x,p) = \left[\cosh(2p) - 1\right]G_1(x) + \sinh(2p) G_2(x).
\end{equation*}
This representation immediately yields that
\begin{equation*}
H_{pp}(x,p) = 4 \cosh(2p)G_1(x) + 4 \sinh(2p) G_2(x),
\end{equation*}
which in turn implies that
\begin{equation*}
H_{pxp}(x,p) = 4H_x(x,p) + 4G_1'(x).
\end{equation*}
As $H_{ppp}(x,p) = 4 H_p(x,p)$, we conclude that the first two terms of the right hand side of \eqref{eqn:entropy_convexity_CW_final_to_prove} equal
\begin{equation}
p H_{pxp}(x,p)H_p(x,p) - pH_{ppp}(x,p) H_x(x,p) = 4 p G_1'(x) H_p(x,p).
\end{equation}
The last term of \eqref{eqn:entropy_convexity_CW_final_to_prove} can be rewritten as
\begin{align*}
\hspace{-1em} & -2 H_{pp}(x,p)\left((\cosh(2p) - 1)G_1'(x) + \sinh(2p)\left(G_2'(x) + 1 - \beta\right) \right) \\
& = -2 H_{pp}(x,p)(\cosh(2p) - 1)G_1'(x) \\
& \qquad + \sinh(2p) H_{pp}(x,p)\left[(1-\beta)(\cosh(\beta x) - 1) + \beta x \sinh(\beta x) \right].
\end{align*}
Rewriting these last two equations, we have to prove for all $x$ and $p \geq 0$ that
\begin{align}
\hspace{-2em} 0 & \leq 8\left[p \sinh(2p) - \cosh^2(2p) + \cosh(2p)\right] G_1(x)G_1'(x) \label{eqn:entropy_convexity_CW_final_bound} \\
\hspace{-2em} & \quad + 8\left[p\cosh(2p) + \sinh(2p)\cosh(2p) - \sinh(2p)\right] G_1'(x)G_2(x)\notag  \\
\hspace{-2em} & \quad + 2 H_{pp}(x,p)\sinh(2p)\left[(1-\beta)(\cosh(\beta x) - 1) + \beta x \sinh(\beta x) \right].\notag 
\end{align}
This will be proven in two steps, first we prove this inequality for $x \geq 0$ and all $p \geq 0$, and afterwards we consider the case that $x \leq 0$ and $p \geq 0$.

\smallskip

\textit{Case 1: $x \geq 0$}. It can immediately be seen that the third line in \eqref{eqn:entropy_convexity_CW_final_bound} is bounded below by $0$. 
For the  first line, we show that
\begin{equation*}
p \mapsto p \sinh(2p) - \cosh^2(2p) + \cosh(2p)
\end{equation*}
is non-positive for $p \geq 0$. First note that $2p \leq \sinh(2p)$, and thus
\begin{align*}
& p \sinh(2p) - \cosh^2(2p) + \cosh(2p) \\
& \quad \leq \frac{1}{2} \sinh^2(2p)- \cosh^2(2p) + \cosh(2p) \\
& \quad = -\frac{1}{2} - \frac{1}{2} \cosh^2(2p) + \cosh(2p) \\
& \quad = -\frac{1}{2}(\cosh(2p) - 1)^2 \\
& \quad \leq 0.
\end{align*}
As $G_1(x)G_1'(x) \leq 0$ for $x \geq 0$ by Lemma \ref{lemma:properties_of_G1G2}, also the first term of \eqref{eqn:entropy_convexity_CW_final_bound} is non-negative.

We proceed with the second term. The map
\begin{equation*}
p\cosh(2p) + \sinh(2p)\cosh(2p) - \sinh(2p)
\end{equation*}
is non-negative for $p \geq 0$ as $\cosh(2p) \geq 1$. Additionally, by Lemma \ref{lemma:properties_of_G1G2}, the product $G_1'(x)G_2(x)$ is non-negative.

We conclude that \eqref{eqn:entropy_convexity_CW_final_bound} holds for $p \geq 0$ and $x \geq 0$.

\smallskip

\textit{Case 2: $x \leq 0$.} The non-negativity for lines 2 and 3 of the right-hand side in \eqref{eqn:entropy_convexity_CW_final_bound} still hold, but we need to show that these lines compensate line 1, that is now negative due to the positivity of the product $G_1(x)G_1'(x)$. In particular, we will show that line three of the right hand side of \eqref{eqn:entropy_convexity_CW_final_bound} compensates the first term. Note that
\begin{equation} \label{eqn:entropy_convexity_CW_G2prime}
0 \geq (1-\beta)(\cosh(\beta x) - 1) + \beta x \sinh(\beta x) = - G_2'(x) - (1-\beta),
\end{equation}
so that the third term of \eqref{eqn:entropy_convexity_CW_final_bound} equals
\begin{align*}
& 2 H_{pp}(x,p)\sinh(2p)\left[(1-\beta)(\cosh(\beta x) - 1) + \beta x \sinh(\beta x) \right] \\
& \quad = -2 H_{pp}(x,p)\sinh(2p)\left[G_2'(x) + (1-\beta)\right] \\
& \quad = -8 \cosh(2p)\sinh(2p) G_1(x)\left[G_2'(x) + (1-\beta)\right] \\
& \qquad - 8 \sinh^2(2p) G_2(x)\left[G_2'(x) + (1-\beta)\right]. 
\end{align*}
By equation \eqref{eqn:entropy_convexity_CW_G2prime} and Lemma \ref{lemma:properties_of_G1G2} the term in the last line is non-negative if $x \leq 0$. Thus, we can use the term in line three to compensate the first term in \eqref{eqn:entropy_convexity_CW_final_bound}. In particular, we have to show that 
\begin{multline*}
0 \leq -8 \cosh(2p)\sinh(2p) G_1(x)\left[G_2'(x) + (1-\beta)\right] \\
+8\left[p \sinh(2p) - \cosh^2(2p) + \cosh(2p)\right] G_1(x)G_1'(x).
\end{multline*}
for $x \leq 0$ and $p \geq 0$. We divide by $8G_1(x) > 0$ and show
\begin{multline*}
0 \leq - \cosh(2p)\sinh(2p) \left[G_2'(x) + (1-\beta)\right] \\
+\left[p \sinh(2p) - \cosh^2(2p) + \cosh(2p)\right] G_1'(x).
\end{multline*}
Below, we will prove that $G_1'(x) + G_2'(x) + (1- \beta) \leq 0$ for $x \leq 0$. Using this inequality, we find
\begin{multline*}
\hspace{-2em} - \cosh(2p)\sinh(2p) \left[G_2'(x) + (1-\beta)\right] + \left[p \sinh(2p) - \cosh^2(2p) + \cosh(2p)\right] G_1'(x) \\
\hspace{-2em} \geq \left[\cosh(2p)\sinh(2p) + p \sinh(2p) - \cosh(2p)(\cosh(2p) - 1)\right] G_1'(x).
\end{multline*}
As $G'_1(x) \geq 0$ for $x \leq 0$ and $p \sinh(2p) \geq 0$ and $\sinh(2p) \geq \cosh(2p) - 1$, we find that this term is non-negative. 

\smallskip

We are left to prove that $G_1'(x) + G_2'(x) + (1-\beta) \leq 0$ for $x \leq 0$. First, we calculate
\begin{align*}
G_1'(x) & = (\beta-1) \sinh(\beta x) - \beta x \cosh(\beta x), \\
G_2'(x) & = (\beta-1) \cosh(\beta x) - \beta x \sinh(\beta x).
\end{align*}
We conclude that
\begin{align*}
& G_1'(x) + G_2'(x) + (1-\beta) \\
& = (\beta - 1) \left[\sinh(\beta x) + \cosh( \beta x) - 1\right] - \beta x \left[\cosh(\beta x) + \sinh(\beta x) \right],
\end{align*}
which yields that $G_1'(x) + G_2'(x) + (1-\beta) \leq 0$ for $x \leq 0$.

\smallskip

We conclude that \eqref{eqn:entropy_convexity_CW_final_bound} holds for all $x \in [-1,1]$ and $p \geq 0$. This implies \eqref{eqn:entropy_convexity_CW_final_to_prove} and thus the entropy-convexity inequality with constant $4(1-\beta)$.

\smallskip

To prove that $4(1-\beta)$ is optimal, we turn to Proposition \ref{assumption:basic_assumption}. In this setting $0$ is the unique stationary point and $-H_{px}(0,0) = 2(1-\beta)$. It follows that $4(1-\beta)$ is optimal for the entropy-information inequality and thus for the entropy-convexity inequality.
\end{proof}

\subsection{Interacting random walks on the hypercube}

For the final example, we use the tensorization results to analyse the trajectories of the empirical distributions of interacting random walks on a hypercube $F = \{-1,1\}^N$.

For the basic model, we consider mean-field interacting walkers on $\{-1,1\}$. For the configuration of $n$ walkers, denoted by $\sigma = \{\sigma(1),\dots, \sigma(n)\} \in \{-1,1\}^n$, we denote the empirical distribution by $\mu \in \cP(\{-1,1\})$ of $\sigma$ by
\begin{equation*}
\mu(\sigma) := \frac{1}{n} \sum_{i=1}^n \delta_{\sigma(i)}.
\end{equation*}
Let $m : \cP(\{-1,1\}) \rightarrow [-1,1]$ denote the magnetization map $m(\nu) = \nu(1) - \nu(-1)$ and denote by $V : [-1,1] \rightarrow \bR$ the function $V(x) = - \frac{1}{2}x^2$.

The random walkers have generator
\begin{equation*}
\cA_{0,n} f(\sigma) = \sum_{j=1}^n e^{-n \beta\left[V(m(\mu^{j})) - V(m(\mu)) \right]} \left[f(\sigma^{j}) - f(\sigma)\right],
\end{equation*}
where $\sigma^{j}$ is the configuration obtained from $\sigma$ by flipping the $j$-th spin and where $\mu^{j}$ is the empirical distribution obtained from $\sigma^{j}$.

Denote by $t \mapsto \sigma_n(t)$ the jump process corresponding to the generator $\cA_{n,0}$. If the distribution of the starting magnetization satisfies the large deviation principle with rate function $I_0$, then the trajectory of the empirical distributions $t \mapsto \mu(\sigma_n(t))$ satisfies the large deviation principle on $D_{\cP(\{-1,1\})}(\bR^+)$ with rate function $I$ given by
\begin{equation*}
I(\gamma) = I_0(\gamma(0)) + \int_0^\infty \cL_0(\gamma(t),\dot{\gamma}(t)) \dd t,
\end{equation*}
if $\gamma$ is absolutely continuous, and $I(\gamma) = \infty$ otherwise. $\cL_0(\mu,v)$ is obtained via the Legendre transform of $H_0(\mu,p)$, given by
\begin{equation*}
H_0(\mu,p) = \mu(1) e^{\beta(\mu(-1) - \mu(1))} \left[e^{p_{-1} - p_1} -1\right] + \mu(-1)e^{\beta(\mu(1) - \mu(-1))} \left[e^{p_{1} - p_{-1}} -1\right].
\end{equation*}
The stationary entropy $S_0$ is given by
\begin{equation*}
S_0(\mu) = \mu(1) \log 2\mu(1) + \mu(2) \log 2 \mu(2) - \frac{1}{2} \beta V(m(\mu)).
\end{equation*}
Note that up to the diffeomorphism $\mu \mapsto \mu(\mu)$ this is exactly the setting of Section \ref{section:Glauber_for_CW}. In particular, we find that we have the entropy convexity inequalities if $\beta \leq 1$ with constant $4(1-\beta)$.

\smallskip

We tensorize these results to random walkers on the hypercube $F = \prod_{i=1}^N \{-1,1\} = \{-1,1\}^N$. A configuration of $n$ walkers is still denoted by $\sigma = \{\sigma(1),\dots,\sigma(n)\}$ but now takes its values in $\{\{-1,1\}^N\}^n$. As above, denote the empirical distribution $\mu \in E := \cP(F)$ of $\sigma$ by
\begin{equation*}
\mu(\sigma) := \frac{1}{n} \sum_{i=1}^n \delta_{\sigma(i)}.
\end{equation*}
Recall from Section \ref{section:tensorization_for_empirical_density_on_product_spaces} the map $\pi_i : \{-1,1\}^n \rightarrow \{-1,1\}$ such that $\sigma$ gets mapped to its $i$-th component. 

The transition operator of $n$ interacting walks on $F$ is given by
\begin{equation} \label{eqn:generator_walks}
\cA_n f(\sigma) =  \sum_{j = 1}^n \frac{1}{N}\sum_{i = 1}^N e^{-n \beta\left[V(m(\pi_i\mu^{i,j})) - V(m(\pi_i\mu)) \right]} \left[f(\sigma^{i,j}) - f(\sigma)\right],
\end{equation}
where $\sigma^{i,j}$ is the configuration obtained from $\sigma$ by flipping the $i$-th coordinate of the $j$-th spin and where $\mu^{i,j}$ is the empirical distribution obtained from $\sigma^{i,j}$.

\smallskip

Note that the rate to flip the $i$-th coordinate depends on the empirical magnetization of the old and the new configuration of the $i$-th coordinate only. This means that, indeed, the system is of product form as in Section \ref{section:tensorization_for_empirical_density_on_product_spaces}. In particular, if $f : \{\{-1,1\}^N\}^n \rightarrow \bR$ is of the form
\begin{equation*}
f(\sigma) = \prod_{i = 1}^N f_i(\pi_i(\sigma)),
\end{equation*}
where the $f_i$ are functions on $\{-1,1\}^n$, we see that
\begin{equation*}
\cA_{n}f(\sigma) = \sum_{i = 1}^N \prod_{k \neq i} f_k(\pi_k(\sigma)) \frac{1}{N}\cA_{n,0}f_i(\pi_i(\sigma)).
\end{equation*}
Thus, considering the large deviations for the trajectories of empirical distributions, we find a Hamiltonian 
\begin{equation*}
H(\mu,p) = \frac{1}{N} \sum_{i=1}^N H_0(\pi_i(\mu),\pi_i(p)),
\end{equation*}
Lagrangian
\begin{equation*}
\cL(\mu,N^{-1}\nu) = \frac{1}{N} \sum_{i=1}^N \cL_0(\pi_i(\mu),\pi_i(\nu))
\end{equation*}
and entropy $S(\mu) = \sum_{i=1}^N S(\pi_i(\mu))$.

We obtain the following theorem as a consequence of Theorem \ref{theorem:interacting_jump_processes_on_two_states_entropy_convexity} and Propositions \ref{proposition:tensorization_for_empirical_measure_product} and \ref{proposition:tensorization_for_empirical_measure_product_trajectories_in_interior}.

\begin{theorem}
Let $\beta \in [0,1]$. Consider $H$ and $S$ as defined above. Then every entropic interpolation satisfies Assumption \ref{assumption:entropic_interpolations_not_on_boundary}.

Additionally, $H$ and $S$ satisfy the entropy convexity inequality with constant $\frac{4}{N}(1-\beta)$. This constant is optimal for the entropy-information and entropy-convexity inequalities.
\end{theorem}

Recalling that in the setting of independent walkers the MLSI($2\kappa$) is equivalent to EII($\kappa$), we recover the result in Example 3.7 of \cite{BoTe06} and Corollary 7.10.(1) of \cite{ErMa12}. Additionally, we see that our entropy convexity inequality holds with the same constant(up to the usual factor 2 difference) as the lower bound on the Ricci curvature in \cite{ErMa12}.

\section{Entropic interpolations remain in the interior for one-dimensional reversible systems} \label{section:proof_entropic_interpolations_in_interior}

To conclude, we prove Proposition \ref{proposition:trajectories_in_interior}. We need some additional results.

\smallskip

To prove that an interpolation $\{x(s)\}_{s \in [0,t]}$ from $a$ to $b$ remains in the interior, we argue by contradiction. Suppose that $x$ that hits the boundary for some $s \in (0,t)$, then we find a cheaper trajectory that also connects $a$ to $b$. To do this, we use the evolution of the entropy $S$ along the interpolation.

We start out with a technical regularity result. 

\begin{lemma} \label{lemma:S_along_flow_is_AC}
In the setting of Proposition \ref{proposition:trajectories_in_interior}, let $\gamma : [0,t] \rightarrow [-1,1]$ be absolutely continuous and such that 
\begin{equation*}
\int_0^t \cL(\gamma(s),\dot{\gamma}(s)) \dd s < \infty.
\end{equation*}
Then $s \mapsto S(\gamma(s))$ is absolutely continuous.
\end{lemma}

Note that this result is non-trivial. A result of Fichtenholz, see Exercise 5.8.61 in \cite{Bo07}, shows that if $DS(x) \rightarrow \infty$ or $DS(x) \rightarrow \infty$ for $x$ close to the boundary, there exists an absolutely continuous trajectory $\gamma$ taking values in $[a,b]$ such that $s \mapsto S(\gamma(s))$ is not absolutely continuous.

\begin{proof}
The proof is somewhat technical and needs the definition of Lusin's property (N). We say that a function $F :(X,\cA,\mu) \rightarrow (Y,\cB,\nu)$ between to measure spaces satisfies (N) if $\nu(F(A)) = 0$ for all $A \in \cA$ with $\mu(A) = 0$.

\smallskip

Pick $\gamma$ that satisfies the assumptions of the lemma. Because $\gamma$ and $S$ are continuous, $s \mapsto S(\gamma(t))$ is continuous. $\gamma$ is absolutely continuous, so it satisfies property (N). As $S$ is continuously differentiable on $(a,b)$ it is absolutely continuous on $(a,b)$. Because $S$ is decreasing in a neighbourhood of $a$ and increasing in a neighbourhood of $b$, the absolute continuity of $S$ on $[a,b]$ follows by the monotone convergence theorem. We conclude that $S$ satisfies (N). Clearly the composition $s \mapsto S(\gamma(s))$ of functions that satisfy (N) also satisfies (N).

\smallskip

To prove that $s \mapsto S(\gamma(s))$ is absolutely continuous, we use Exercise 5.8.57 of \cite{Bo07} that states that a continuous function $f :[\alpha,\beta] \rightarrow \bR$ with property (N) is absolutely continuous if there exists a Lebesgue integrable function $g$ such that $f'(x) \leq g(x)$ at almost every point where $f'(x)$ exists.

\smallskip

We show that we can find such a function $g$ for $f(x) := S(\gamma(s))$, using the assumption that the Lagrangian cost of the trajectory is finite.

First of all, $\gamma$ is differentiable at almost every time. Thus, for almost every time $s$ for which $\gamma(s) \in (a,b)$, the map $f$ is differentiable. For such $s$, we have by Lemma \ref{lemma:derivative_of_entropy} that
\begin{equation*}
\frac{\dd}{\dd s} S(\gamma(s)) \leq \cL(\gamma(s),\dot{\gamma}(s)).
\end{equation*}
As $S$ has its maxima at the boundary, a time $s$ for which $\gamma(x) \in \{-1,1\}$ and $f$ is differentiable, must satisfy $f'(s) = 0 \leq \cL(\gamma(s),\dot{\gamma}(s))$.

\smallskip

Thus, for almost every time $s$ for which $s \mapsto S(\gamma(s))$, we have $\frac{\dd}{\dd s}S(\gamma(s)) \leq \cL(\gamma(s),\dot{\gamma}(s))$. By the assumption of the lemma and Exercise 5.8.57 of \cite{Bo07}, we conclude that $s \mapsto S(\gamma(s))$ is absolutely continuous.
\end{proof}

Our second auxiliary result is a decomposition for $\cL$, which is a result also obtained e.g. in \cite{MPR13}. The decomposition there is given in terms of $\Psi, \Psi^*$ and the decomposition is used to interpret the solution of the McKean-Vlasov equation as the flow that optimizes an entropy-dissipation inequality. Here we give a different interpretation of this decomposition. We first introduce a tilted Hamiltonian $H[x] : E^\circ \times \bR^d \rightarrow \bR$ by
\begin{equation} \label{eqn:def_tilted_hamiltonian}
H[x](y,p) = H(y,p + \frac{1}{2}DS(x)) - H(y,\frac{1}{2}DS(x)).
\end{equation} 
$x$ is a stationary point of the McKean-Vlasov dynamics associated to $H[x]$, i.e. $H[x]_p(x,0) = 0$. Define $\cL[x]$ to be the Lagrangian associated to $H[x]$:
\begin{equation*}
\cL[x](y,v) = \sup_p \left\{ pv - H[x](y,p)\right\}.
\end{equation*}
In relation to the decomposition in \cite{MPR13}, note that $V = 2S$, $\Psi(x,v) = \cL[x](x,v)$ and $\Psi^*(x,p) = H[x](x,p)$. In the special cases that $p=- \frac{1}{2}DS(x)$, we find additionally that
\begin{equation*}
\Psi^*\left(x,- \frac{1}{2}DS(x)\right) = H[x]\left(x,- \frac{1}{2}DS(x)\right) = \cL(x,0).
\end{equation*}

Thus, we find that $\cL(x,v)$ can be decomposed into a cost for making $x$ a stationary point, the cost for having speed $v$ under the tilted dynamics and a correction term: one-half the increase of entropy along the flow.

\begin{lemma} \label{lemma:decomposition_Lagrangian}
For $x \in (a,b)$ and $v \in \bR$, we have the decomposition
\begin{equation*}
\cL(x,v) = \cL(x,0) + \cL[x](x,v) + \frac{1}{2}\ip{DS(x)}{v},
\end{equation*}
where $\cL[x](y,v)$ is defined as the Legendre transform of $p \mapsto H[x](y,p)$, as defined in \eqref{eqn:def_tilted_hamiltonian}. 

\smallskip

Furthermore, for $x \in (a,b)$, we have $\cL[x](x,v) = \cL[x](x,-v)$. Finally, for any absolutely continuous trajectory $\gamma : [0,t] \rightarrow [a,b]$ that has finite Lagrangian cost, we have
\begin{equation*}
\int_0^t \cL(\gamma(s),\dot{\gamma}(s)) - \cL(\gamma(s),-\dot{\gamma}(s)) \dd s = S(\gamma(t)) - S(\gamma(0)).
\end{equation*}
\end{lemma}

\begin{proof}
The first two claims follows as in Lemma 2.1 and Proposition 2.1 \cite{MPR13}.

\smallskip

For the final claim, note that $s \mapsto S(\gamma(s))$ is absolutely continuous by Lemma \ref{lemma:S_along_flow_is_AC}. For times $s$ that $\gamma(s) \in (a,b)$, the derivative of $s \mapsto S(\gamma(s))$ is given by Lemma \ref{lemma:derivative_of_entropy}, using that $H = H^*$. For almost all times $s$ such that $\gamma(s) \notin (a,b)$, $s \mapsto S(\gamma(s))$ is differentiable as the map is absolutely continuous. For these times the derivative must be $0$ as $S$ has its (strict) maxima on the boundary. For these times, note that also $\cL(x,0) - \cL(x,0) = 0$. We conclude that the final claim follows by integration.
\end{proof}

We conclude this section by proving that all entropic interpolations remain in the interior of $E$.

\begin{proof}[Proof of Proposition \ref{proposition:trajectories_in_interior}]
Fix $t > 0$ and $\alpha,\beta \in E$. Let $\gamma$ be an optimal trajectory such that $\gamma(0) = \alpha$ to $\gamma(t) = \beta$.

\smallskip

The strategy of the proof is as follows. We argue by contradiction. First we assume that there exists an interval $[t_0,t_1] \subseteq [0,t]$ on which the trajectory is on the boundary of $E$. Then, we construct a new trajectory, which is on the boundary for the times $t_0$ and $t_1$, but not for $s \in (t_0,t_1)$, which has lower cost. This contradicts the assumption that our trajectory was optimal.
As a second step, we assume there is an isolated time $t^* \in (0,t)$ for which the trajectory is on the boundary. In this setting, we construct a compatible trajectory that remains in the interior for an interval $(t_-,t_+) \ni t^*$ with lower cost, again contradicting the assumption that our trajectory was optimal.

\smallskip

These two contradictions show that an optimal trajectory can not be on the boundary for a time $s \in (0,t)$.

\smallskip

First assume that there exists an interval $[t_0,t_1]$, $t_0 \neq t_1$ such that the optimal trajectory $\gamma$ satisfies $\gamma(s) = a$ for $s \in [t_0,t_1]$. The argument for the boundary $b$ is similar. We construct $\gamma^*$ that has a lower cost to obtain a contradiction. Fix some $\varepsilon > 0$ small enough such that $\varepsilon < \frac{1}{2}(t_1 - t_0)$ and such that the solution of $\dot{x} = H_p(x,0)$ started at $x_0 = a$ does not leave $U_{a}$. Note any solution $\{x(t)\}_{t \geq 0}$ of the McKean-Vlasov equation $\dot{x} = H_p(x,0)$ satisfies $x(t) \in (a,b)$ for $t > 0$ by assumption (b) of the Proposition.

\smallskip

Define $\gamma_\varepsilon : [t_0,t_1] \rightarrow E$ as $\gamma_\varepsilon(0) = a$, $\dot{\gamma}_\varepsilon(s) = H_p(\gamma_\varepsilon(s),0)$ for $s \leq \varepsilon$ and $\gamma_\varepsilon(s) = \gamma_\varepsilon(\varepsilon) =: z(\varepsilon)$ for $s \in [\varepsilon,\frac{1}{2}]$. Additionally, we set $\gamma_\varepsilon$ to be the time-reversed trajectory on the second half of the interval: $\gamma_\varepsilon(t_0 + s) = \gamma_\varepsilon(t_1 - s)$. 

\smallskip

Splitting $[0,t]$ into the two symmetric parts, applying the final part of Lemma \ref{lemma:decomposition_Lagrangian} on the non-stationary part of $\gamma$, we find
\begin{multline*}
\int_0^t \cL(\gamma_\varepsilon(s),\dot{\gamma}_\varepsilon(s)) \dd s = 2\int_0^\varepsilon \cL(\gamma_\varepsilon(s),\dot{\gamma}_\varepsilon(s)) \dd s \\
+ (t_2 - t_1 - 2 \varepsilon) \cL(z(\varepsilon),0) + \left(S(-1)  - S(z(\varepsilon))\right).
\end{multline*}
Now the first term on the right-hand side is $0$ as $\dot{\gamma}_\varepsilon(s) = H_p(\gamma_\varepsilon(s),0)$ for $s \leq \varepsilon$, thus
\begin{multline*}
\int_0^t \cL(\gamma_\varepsilon(s),\dot{\gamma}_\varepsilon(s)) - \cL(\gamma(s),\dot{\gamma}(s)) \dd s \\
\hspace{-1em} = (t_2 - t_1)\left(\cL(z(\varepsilon),0) - \cL(a,0)\right) - 2 \varepsilon \cL(z(\varepsilon,0)) + S(a) - S(z(\varepsilon)).
\end{multline*}
The middle term on the right hand is non-negative. That the first and the third term combined are non-negative for small $\varepsilon$ follows from assumption (d) of the proposition.

Thus, we have contradicted the assumption that there exists an interval $[t_0,t_1]$, $t_0 \neq t_1$ such that $\gamma$ satisfies $\gamma(s) = a$ for $s \in [t_0,t_1]$.

\smallskip

Now suppose there exists $t^* \in (0,t)$ such that $\gamma(t^*) = a$. We show that this leads to a  contradiction. Fix $z > a$. Then the set $B_z : = \gamma^{-1}([a,z))$ is open in $[0,T]$. Because an open set in $\bR$ is the countable disjoint union of open intervals by the Lindel\"{o}f lemma, there are three possibilities:
\begin{enumerate}[(a)]
\item $t^* \in (t_-,t_+)$, $(t_-,t_+) \subseteq B_z$, $t_-,t_+ \notin B_z$,
\item $t^* \in (t_-,t]$, $(t_-,t] \subseteq B_z$, $t_- \notin B_z$,
\item $t^* \in [0,t_+)$, $[0,t_+) \subseteq B_z$, $t_+ \notin B_z$.
\end{enumerate}
Clearly, if $(b)$ happens for all $z > a$, then $[t^*,t] \subseteq \gamma^{-1}(a)$ which contradicts the conclusion of the first part of the proof. A similar contradiction occurs for (c). Note that in case of (a), we have that $\gamma(t_-) = \gamma(t_+) = z$ by the continuity of $\gamma$.

\smallskip

Thus, we can choose $z > a$ close enough to $a$, such that (a) is satisfied and such that $[a,z) \subseteq U_{a}$. Again we construct a cheaper trajectory $\gamma^*$. As noted above, there are $0 < t_- < t^* < t_+ < t$ such that we have $\gamma(t_-) = \gamma(t_+) = z$ and $a \leq \gamma(s) \leq z$ for $t_- \leq s \leq t_+$. Consider the trajectory
\begin{equation*}
\gamma_z(s) =
\begin{cases}
\gamma(s) & \text{for } s \notin [t_-,t_+], \\
z & \text{for } s \in [t_-,t_+].
\end{cases}
\end{equation*}
Using Lemma \ref{lemma:decomposition_Lagrangian}, integrating over time in $[t_-,t_+]$, we find
\begin{equation} \label{eqn:reformulation_of_cost_integral}
\int_{t_-}^{t_+} \cL(\gamma_z(s),\dot{\gamma}_z(s)) \dd s = \int_{t_-}^{t_+} \cL(z,0) \dd s
\end{equation}
and
\begin{equation*}
\int_{t_-}^{t_+} \cL(\gamma(s),\dot{\gamma}(s)) \dd s = \int_{t_-}^{t_+} \cL[\gamma(s)](\gamma(s),\dot{\gamma}(s)) + \cL(\gamma(s),0) \dd s.
\end{equation*}
As the first term of the integrand on the right is non-negative, and second term in the integrand is bounded from below by the integrand on the right in \eqref{eqn:reformulation_of_cost_integral} by condition (c) of the proposition, we find that $\gamma_z$ has a lower cost than $\gamma$, contradicting the assumption that $\gamma$ was optimal.

\smallskip

We conclude that an optimal trajectory can only attain a boundary point at its initial or final time.
\end{proof}

\smallskip

\textbf{Acknowledgement}
The author thanks Michiel Renger for useful discussion on the topic of the paper. The author is supported by The Netherlands Organisation for Scientific Research (NWO), grant number 600.065.130.12N109.


\bibliographystyle{plain} 
\bibliography{../KraaijBib}{}

\end{document}